\documentclass[a4paper]{amsart}
\usepackage[english]{babel}
\usepackage[T1]{fontenc}
\usepackage[format=plain, font=footnotesize]{caption}

\usepackage[top=3cm,bottom=2cm,left=3cm,right=3cm,marginparwidth=1.75cm]{geometry}
\usepackage{amsmath}
\usepackage{graphicx}%
\usepackage{amsfonts}
\usepackage{amssymb}
\usepackage{comment}
\usepackage{hyperref}
\hypersetup{
     colorlinks = true,
     linkcolor = blue,
     anchorcolor = blue,
     citecolor = red,
     filecolor = blue,
     urlcolor = blue
     }
\usepackage{pstricks,xy,pst-node}
\xyoption{all}
\usepackage{amsthm}
\usepackage{longtable}
\usepackage{caption,subcaption}
\usepackage{booktabs} %to produce nicer tables
\usepackage[color]{changebar} %to introduce changebars
\cbcolor{red} %the color of the changebars
\usepackage{mathrsfs} %to use \mathscr
\usepackage{verbatim} %to use \verb
\usepackage{enumitem}
\setlist[enumerate]{label=\arabic*.}

\usepackage{tikz}
\usepackage{tikzcd}

\usepackage{amssymb}
\usepackage{tikz-cd}
\usetikzlibrary{cd}
\newcommand{\N}{\mathbb{N}}

\newcommand{\C}{\mathbb{C}}

\newcommand{\R}{\mathbb{R}}
\newcommand{\EQUATION}{equation}
\newcommand{\beq}{\begin{\EQUATION}}
\newcommand{\eeq}{\end{\EQUATION}}

\newtheorem{theorem}{Theorem}
\newtheorem{lemma}[theorem]{Lemma}

\newtheorem{proposition}[theorem]{Proposition}

\theoremstyle{definition}

\newtheorem{definition}[theorem]{Definition}

\newtheorem{remark}[theorem]{Remark}
\newtheorem{example}[theorem]{Example}

\newcommand{\CC}{C}

\usepackage{mathtools} %scommentare per far numerare solo le equazioni che sono chiamate nel testo
\mathtoolsset{showonlyrefs,showmanualtags} %scommentare per far numerare solo le equazioni che sono chiamate nel testo

\title{Quantitative singularity theory for random polynomials}
\author{Paul Breiding, Hanieh Keneshlou and Antonio Lerario}
\date{\today}
\thanks{Technische Universit\"at Berlin, Stra{\ss}e des 17. Juni 136, 10623 Berlin, Germany. \texttt{breiding@math.tu-berlin.de}}
\thanks{Max-Planck Institute for Mathematics in Sciences, Inselstr.22, 04103 Leipzig, Germany. \texttt{keneshlo@mis.mpg.de}}
\thanks{Scuola Internazionale Superiore di Studi Avanzati, Via Bonomea 265, 34136 Trieste, Italy. \texttt{lerario@sissa.it}}

\subjclass[2010]{}

\begin{document}
\maketitle
\begin{abstract}
Motivated by Hilbert's 16th problem we discuss the probabilities of topological features of a system of random homogeneous polynomials. The distribution for the polynomials is the Kostlan distribution.
The topological features we consider are \emph{type-$W$ singular loci}. This is a term that we introduce and that is defined by a list of equalities and inequalities on the derivatives of the polynomials. In technical terms a type-$W$ singular locus is the set of points where the jet of the Kostlan polynomials belongs to a semialgebraic subset $W$ of the jet space, which we require to be invariant under orthogonal change of variables. For instance, the zero set of polynomial functions or the set of critical points fall under this definition.

We will show that, with overwhelming probability, the type-$W$ singular locus of a Kostlan polynomial is ambient isotopic to that of a polynomial of lower degree. As a crucial result, this implies that complicated topological configurations are rare. Our results extend earlier results from Diatta and Lerario who considered the special case of the zero set of a single polynomial.
Furthermore, for a given polynomial function $p$ we provide a deterministic bound for the radius of the ball in the space of differentiable functions with center $p$, in which  the $W$-singularity structure is constant. \end{abstract}
\section{Introduction}
Hilbert's 16th problem was posed by David Hilbert at the Paris ICM in 1900 and, in its general form, it asks for the study of the maximal number and the possible arrangement of the components of a generic real algebraic hypersurface of degree $d$ in real projective space. Since Hilbert had posed his question, many mathematicians have contributed to the subject: Hilbert \cite{Hilbert1891},  Rohn \cite{Rohn1913}, Petrovsky \cite{Petrovsky1938},
Rokhlin \cite{Rokhlin1978},
Gudkov \cite{Gudkov1978}, Nikulin \cite{Nikulin1980}, Viro \cite{Viro1982, Viro1986, Viro2008}, Kharlamov \cite{Kharlamov1978,Kharlamov1981,Kharlamov1984}, and more.

Hilbert's problem not only concerns the topology of the hypersurface but also the way it is embedded inside the projective space. The difference between these two sides of the problem can be illustrated by considering the sextic polynomials $P_1(x_0,x_1,x_2)=(x_1^4 +x_2^4 -x_0^4)(x_1^2 +x_2^2 -2x_0^2)+x_1^5x_2$  and $P_2(x_0,x_1,x_2)=10(x_1^4 +x_2^4 -x_0^4)(x_1^2 +x_2^2 -2x_0^2)+x_1^5x_2$ and and their zero sets, which are shown in Figure \ref{fig1}. Both of them have two connected components and so their topological types agree. But their rigid isotopy types are different, because one cannot move, in the projective plane, the inner oval on the left picture outside without crossing the outer oval. Being rigidly isotopic means that $P_1$ and $P_2$ belong to the same connected component of $\mathbb{R}[x_0, x_1, x_2]_{(6)}\backslash \Sigma$, where $\Sigma$ denotes the set of singular curves.

\begin{figure}[ht]
\begin{center}
\includegraphics[height = 3.2cm]{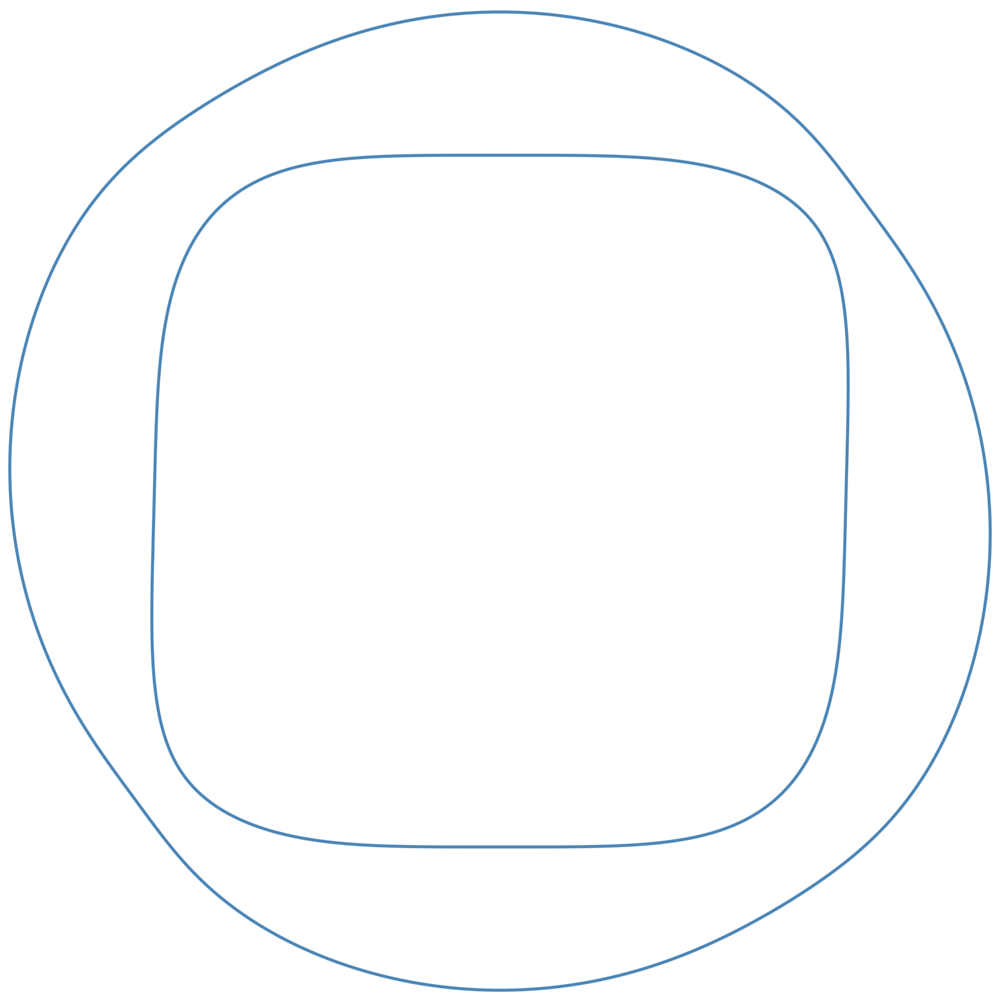}\hspace{3cm}\includegraphics[height = 3.2cm]{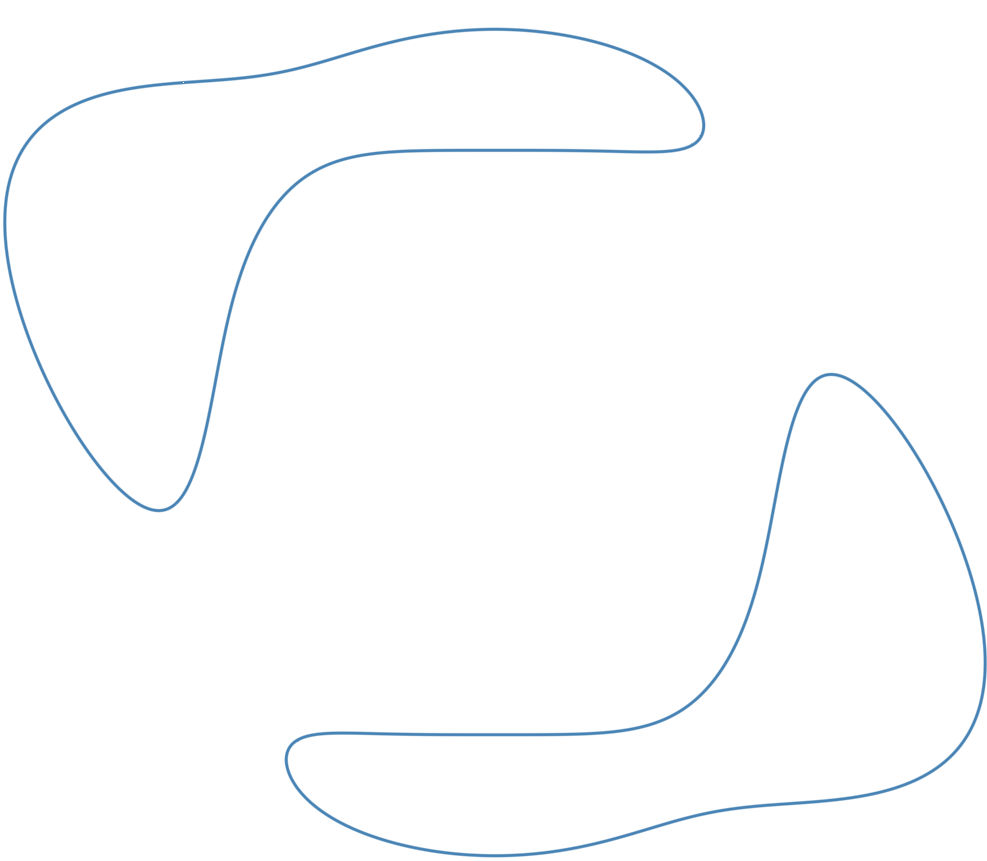}
\caption{The zero sets of $P_1$ (on the left) and $P_2$ (on the right) plotted on the affine chart $x_0=1$.\label{fig1}}
\end{center}
\end{figure}

In this article we approach this classical topic from a probabilistic point of view.
That is, we assume a probability distribution on the space of polynomials and consider \emph{statistical} properties of topological configurations. Moreover, we do not only consider the topology of zero sets. In fact, the case of the zero set of a single polynomial was already considered in \cite{Antonio}. Rather, we consider \emph{type-$W$ singular loci}. We give a rigorous definition below. Among others, those singular loci include
\begin{enumerate}[itemsep = 4pt]
\item[(1)] The zero set of $p:S^n\to \R^k$ in the unit sphere $S^n$.
\item[(2)] The set of critical points of $p:S^n\to \R$ on $S^n$.
\item[(3)] The set of nondegenerate {minima} of $p:S^n\to \R$ on $S^n$.
\item[(4)] The set of points where a polynomial map $p=(p_1, p_2):S^2\to \R^2$ has a Whitney cusp.
\end{enumerate}
Notice that in this list we have switched from the real projective space to the unit sphere. The reason is that polynomials define functions on the sphere, but (unless the degree is even) they do not define functions on the projective space. For this reason, in the following, we will exclusively consider loci inside the sphere $S^n$ and not in projective space; however we observe that, since the latter is double covered by the former, the study of the spherical case can be related to the projective one using standard algebraic topology techniques.

In this article, we follow the same philosophy of \cite{Antonio} and we focus on \emph{tail probabilities}. We want to show that a system of $m$ \emph{Kostlan polynomials} $(P_1,\ldots, P_m)$ rarely has a set of point from the list above that has ``complicated topology''. A kostlan polynomial of degree $d$ is defined as
\begin{equation}\label{Kostlan}
P_k(x_0,\ldots,x_n) = \sum_{\alpha_0+\cdots+\alpha_n=d} \xi_{\alpha_0,\ldots,\alpha_n}^{(k)} \,\sqrt{\tfrac{d!}{\alpha_0!\cdots \alpha_n!}}\, x_0^{\alpha_0}\cdots x_n^{\alpha_n},
\end{equation}
where the $\xi_{\alpha_0,\ldots,\alpha_n}^{(k)} $ are i.i.d.\ standard Gaussian random variables. By ``complicated topology'' we mean configurations that can't be realized by polynomials of lower degree.

We show in Theorem \ref{thm:main} below that with high probability the type-$W$ singular locus of a system of Kostlan polynomials $(P_1,\ldots, P_m)$ with maximal degree $d:=\max\{\deg(P_1),\ldots,\mathrm{deg}(P_m)\}$ is \emph{ambient isotopic} to the singular locus of a system of polynomials of degree approximatively $\sqrt{d\log d}$. By this we mean that there exists a continuous family of diffeomorphisms $\varphi_t:S^n\to S^n$ with $\varphi_0=\textrm{id}_{S^n}$ that at time $t=1$ brings the singular locus of the first system to the singular locus of the second one. The notion ambient isotopy is weaker than the notion rigid isotopy that Hilbert used. However, we can't work with rigid isotopies in our setting, because this is not defined for pairs of polynomials that live in different spaces -- on the one hand, polynomials of degree $d$ and on the other hand polynomials of degree $ \sqrt{d\log d}$. We need to compare those polynomials in the space of all $C^\infty$ functions!

In particular, our results also imply that \emph{maximal configurations}, i.e. type-$W$ singularities that can't be realized by polynomials of lower degree, have exponential small probability as $d\to\infty$. They are virtually non existent under the Kostlan distribution. This has implications for numerical experiments: for large $d$ it is impossible to find maximal configurations by sampling Kostlan polynomials. For zero sets this was already observed in \cite{KKPSS2017}.

Sarnak \cite{Sarnak} suggested using the probabilistc point of view in 2011. Since then the research area has seen much progress \cite{GaWe1, GaWe3, GaWe2, FLL, NazarovSodin1, NazarovSodin2, Sarnak, SarnakWigman,Lerariolemniscate, Letwo, Lerarioshsp, LeLu:gap, Antonio, LerarioStecconi}. Today, the merging of algebraic geometry and probability theory goes under the name of \emph{Random Algebraic Geometry}. The motivation behind taking a statistical point of view is that already for curves in the plane the number of rigid isotopy types grows super-exponentially as the degree of the curve goes to infinity \cite{OrevkovKharlamov}. Therefore, a deterministic enumeration of all possibilities is a hopeless endeavor. On the other hand it was shown in \cite{Antonio} that for \emph{Kostlan}  curves there are only few types that appear with significantly high probability. This result motivated the more general study in this article.

\subsection{The Kostlan distribution} Our choice of the Kostlan distribution (\ref{Kostlan}) has several reasons. The first is that Kostlan polynomials are invariant under orthogonal change of variables: if $P$ is a Kostlan polynomial in $n+1$ variables, then for any orthogonal $(n+1)\times (n+1)$-matrix $R$ we have $P \circ R \sim P$. We believe that a reasonable probability distribution should have this property, since we are considering topological features of \emph{geometric sets}. Following Klein \cite{Klein} those should be defined being invariant under orthogonal change of coordinates of the ambient space. However, Kostlan polynomials are not the only orthogonally invariant probability distribution. We need more reasons to justify this choice. A first possible one is that the Kostlan distribution is particularly suited for comparing real algebraic geometry with complex algebraic geometry: in fact if one considers complex Kostlan polynomials (defined by taking complex Gaussians in \eqref{Kostlan}), the resulting distribution is the unique gaussian distribution (up to multiples) which is invariant under unitary change of coordinates. Ultimately, this connection is why random real algebraic geometry behaves as the ``square root'' of complex algebraic geometry, see \cite{LerarioStecconi}. Moreover, up to multiples, the Kostlan distribution is the unique, among the orthogonally invariant ones, for which we can write a random polynomial as a linear combination of the monomials with independent gaussian coefficients -- thus it is ``simple'' to write a Kostlan polynomial.
Finally, another reason is that among the orthogonally invariant probability distributions on the space of polynomials the Kostlan distribution behaves well under a certain projection, which is be the key part in the proof of our main result Theorem \ref{thm:main} on the tail probabilities.

\subsection{Singularities}\label{sec:singularities}
The examples of singular loci above can be summarized under the following technical definition involving the \emph{jet space} $J^r(S^n, \R^m)$. We recall the precise definition of the jet space in Section~\ref{sec:jet_space}. One may think of points in $J^r(S^n, \R^m)$ as lists of derivatives of functions at a point. Those lists are called \emph{jets}. In this interpretation, each function $f\in C^r(S^n, \R^m)$ defines a map $j^r f :S^n\to J^{r}(S^n, \R^m)$, called the $r$-th jet prolongation, that maps $x$ to the list of derivatives of $f$ at $x$. The precise definition of this is given in Definition \ref{r_jet_prolongation} below.

The key aspect is that $J^r(S^n, \R^m)$ has a natural semialgebraic structure and we can therefore define semialgebraic sets therein.

\begin{definition}[The type-$W$ singular locus]\label{def:st}
We call a subset $W\subseteq J^{r}(S^n, \R^m)$ a \emph{singularity type}, if it is semialgebraic and invariant under diffeomorphisms induced by orthogonal change of variables. Given $f\in C^r(S^n, \R^m)$, the subset $U = j^r f^{-1}(W)\subseteq S^n$ is called the type-$W$ singular locus of $f$.
\end{definition}
The semialgebraic sets describing the above examples are as follows.
\begin{enumerate}[itemsep = 4pt]
\item[(1)] $W=S^n\times \{0\}\subset J^{0}(S^n, \R^m)$.
\item[(2)] $W=\{\mathrm{d}f=0\}\subset J^1(S^n, \R)$.
\item[(3)] $W=\{\mathrm{d}f=0, \mathrm{d}^2f>0\}\subset J^2(S^n, \R)$.
\item[(4)] $W\subset J^3(S^2, \R^2)$ gives conditions on the derivatives of $f:S^2\to \R^2$ up to order three, such that in some local coordinates $f$ has the form $(x_1, x_2)\mapsto(x_1, x_2^3-x_1x_2)$ (see \cite{Callahan}).
\end{enumerate}

\begin{definition}[Ambient isotopic pairs]\label{def:ad}Let $C_0, C_1$ be stratified subcomplexes of the sphere $S^n$. We say that the two pairs are \emph{ambient isotopic}, denoted
\beq (S^n, C_0)\sim (S^n, C_1),\eeq
if there exists a family of diffeomorphisms $(\varphi_t:S^n\to S^n)_{0\leq t\leq 1}$ with $\varphi_0=\mathrm{id}_{S^n}$ and $\varphi_1(C_0)=C_1.$
\end{definition}

\subsection{Organization of the article}
The rest of the article is now organized as follows. In the next section we state our main results. In Section \ref{sec:jet} we  recall the definition of jet space and use it for defining the discriminant locus of a singularity type. Then, in Section \ref{sec:norms_and_polys} we recall the definition of harmonic polynomials, the decomposition of the space of polynomials into the harmonic basis, and define several norms for polynomials, which will be used in the next sections. In Section \ref{sec:proof_thm_2} we prove our result on quantitative stability from Theorem \ref{thm:quantstab} and in Section~\ref{sec:low_degree_approx} we prove Theorem \ref{thm:main}. Finally, in Section \ref{sec:maximal_is_rare} we discuss what our results imply for maximal configurations.

\subsection{Acknowledgements}
This article was written in parts while the third author was on a visiting fellowship of the Max-Planck Institute for Mathematics in the Sciences in Leipzig. The authors also wish to thanks the anonymous referees, whose comments and suggestions improved the structure of the paper.

\section{Main Results}
Having clarified in the previous section what we mean by type-$W$ singular locus of a map, we can now state our main results.
\subsection{Main Result 1: Low-Degree Approximation}
The first result essentially tells that most real singularities given by polynomial equations of degree $d$ are ambient isotopic to singularities given by polynomials of degree $O(\sqrt{d \log d})$.  This means that for any singularity type $W$ the probability of the following event goes to one as $d\to \infty$: Let $p:S^n\to \R^m$ be given as the restriction to the sphere of a system of Kostlan polynomials. There exists a
polynomial $q$ of degree $O(\sqrt{d\log d})$ such that $(S^n, j^rq^{-1}(W))\sim (S^n, j^rp^{-1}(W))$. The new polynomial $q$ can be thought as a low-degree approximation of $p$.

The approximation procedure is constructive in the sense that one can read the approximating polynomial $q$ from a linear projection of the given one. It is also quantitative in the sense that the approximating procedure will hold for a subset of the space of polynomials with measure increasing very quickly to full measure as the degree goes to infinity.

To be more specific, we denote by $\mathcal{P}_{n,d}=\mathbb{R}[x_0, \ldots, x_n]_{(d)}$ the space of homogeneous polynomials of degree $d$. We  recall  from Section \ref{sec:harmonics} below that this space admits a decomposition:
\begin{equation}\label{eq:dec}\mathcal{P}_{n,d}=\bigoplus_{d-\ell\in 2\mathbb{N}}\|x\|^{d-\ell}\mathcal{H}_{n,\ell}\end{equation}
where $\mathcal{H}_{n, \ell}$ denotes the space of homogeneous, harmonic polynomials of degree $\ell$. Given $P\in \mathcal{P}_{n,d}$, we denote by $p:S^n\to \mathbb{R}$ its restriction to the unit sphere and for $L\in \{0, \ldots, d\}$ we define  $p|_{L}=\sum_{\ell\leq L,\, d-\ell\in 2\mathbb{N}} p_\ell,$
where $p_\ell$ is the restriction to $S^n$ of the polynomial appearing in the decomposition $P=\sum_{\ell}P_\ell$ given by \eqref{eq:dec}.
Given polynomials $p_1, \ldots, p_m$ with $\deg(p_i)=d_i$, we can form the polynomial map $p=(p_1, \ldots, p_m):S^n\to \mathbb{R}^m$ and for $L\in \{0, \ldots, d\}$ we can define:
\begin{equation}\label{eq:projmap} p|_L=(p_1|_{L}, \ldots, p_m|_{L}).\end{equation}
We will denote by $\underline{d}=(d_1, \ldots, d_m)$ and by $d=\max d_i$.
\begin{definition}[Low-degree approximation event]\label{def:stabilityevent} Let $W\subseteq J^{r}(S^n, \mathbb{R}^m)$ be a singularity type. Let $p=(p_1, \ldots, p_m)$ be a random polynomial map and $L\in \{0, \ldots, d\}$ . We can consider \begin{equation}A_L=\bigg\{(S^n, j^{r}p|_{L}^{-1}(W))\sim( S^n, j^{r}p^{-1}(W))\bigg\}\subset \mathcal{P}_{n, \underline{d}}\end{equation}
(i.e. the type-$W$ singularities of $p$ and $p|_L$ are ambient isotopic). We call $A_L$ the \emph{low-degree approximation event} for type-$W$ singularities.
\end{definition}

Here is our first main theorem.

\begin{theorem}[Low-degree approximation of type-$W$ singular locus]\label{thm:main}
Let $W\subseteq J^{r}(S^n, \R^m)$ be a singularity type (as defined in Definition \ref{def:st}). Let $P=(P_1, \ldots, P_m)$ be a system of
homogeneous Kostlan polynomials in $n+1$ many variables and of degrees $\underline{d}=(d_1, \ldots, d_m)$, and denote by $d=\max d_i$. Let $p=P|_{S^n}$ be the restriction of $P$ to the sphere $S^n$. For an integer $0\leq L\leq d$ let us denote by $A_L$ the event that the type-$W$ singular locus of $p$ and $p|_{L}$ are ambient isotopic (in the sense of Definition \ref{def:ad}).
Then, we have the following behavior for three different regimes of $L$:
\begin{enumerate}[itemsep = 4pt]
\item There exists $b_0>0$ such that for all $b\geq b_0$ there exist $a_1, a_2>0$ with the property that, choosing $L=b \sqrt{d\log d}$ we have $\mathbb{P}(A_L)\geq 1-a_1d^{-a_2}$.
\item For every $\tfrac{1}{2}<b<1$, there exist $a_1, a_2>0$ (with $0<a_2<1$), such that choosing $L=d^b$ we have $\mathbb{P}(A_L)\geq 1-a_1e^{-d^{a_2}}$.
\item For every $0<b<1$ there exist $a_1, a_2>0$ such that choosing $L=bd$ we have that $\mathbb{P}(A_L)\geq 1-a_1 e^{-a_2d}$.
\end{enumerate}
Conversely:
\begin{enumerate}[itemsep = 4pt, resume]
\item For all $a>0$, there exists $b>0$, such that choosing $L=b\sqrt{d\log d}$ we have $\mathbb{P}(A_L)\geq 1-d^{-a}$ for $d$ large enough.
\item For all $0<a<1$, there exists $0<b<1$, such that choosing $L= d^b$ we have for large enough $d$ that $\mathbb{P}(A_L)\geq 1-e^{-d^a}$.
\item For all $a>0$, there exists $b>0$, such that choosing $L= bd$ we have $\mathbb{P}(E_L)\geq 1-e^{-ad}$ for~$d$ large enough.
\end{enumerate}
\end{theorem}

We call the second and the third regime in the theorem ``exponential rarefaction of maximal configurations'' because the result tells that maximal configurations, i.e. polynomials of degree~$d$ whose singular loci are not ambient isotopic to that of a polynomial of smaller degree, have exponentially small probability in the space of all polynomials. In order to get an insight of these results, we will provide some applications of them for the explicit case of the singularities described in Section \ref{sec:singularities}, in the last section of this paper.

\subsection{Main Result 2: Stable Neighborhoods and Quantitative Stability}
A key step in proving Theorem \ref{thm:main} is proving a deterministic bound on the size of the \emph{stable neighborhood} of $P$. In order to state the result we first give two definitions.
\begin{definition}\label{def:transversality}
Let $f\in C^{r+1}(S^n, \R^m)$, and $W\subseteq J^{r}(S^n,\R^m)$ be a singularity type. We fix a semialgebraic stratification $W=\coprod_{i=1}^k W_i$ into smooth semialgebraic strata. The $r$-jet map $ j^{r}f$ is called transversal to the stratum $ W_i $ if for all $ x\in S^n$ we either have $  j^{r}f(x)\notin W_i $, or
$$  j^{r}f(x)\in W_i\quad \text{and}\quad d_{x}(j^{r}f)(T_xS^n) +T_{j^{r}f(x)}W_i= T_{j^{r}f(x)}J^{r}(S^n,\R^m).$$
Here, $d_{x}(j^{r}f)$ is the differential of $j^rf$ (i.e. the induced map at the level of the tangent spaces).
The $r$-jet map $ j^{r}f$ is called transversal to $ W $ if it is transversal to all the strata of $W=\coprod_{i=1}^k W_i$. We write
$$ j^{r}f\pitchfork W $$
when $ j^{r}f $ is transversal to $ W $.
\end{definition}

Observe that we need $f\in C^{r+1}$ in order to talk about transversality of its $r$-jet to $W$: in fact, if $f\in C^{r+s}(S^n, \R^m)$, the jet extension $j^rf:S^n\to J^{r}(S^n, \R^m)$ is of class $C^s$ (see \cite[p.\ 61]{Hirsch}). Therefore, since the  transversality condition involves the differential of $j^rf$, we need this map to be at least $C^1$, i.e. $f$ to be at least $C^{r+1}$.

\begin{definition}[Stable neighborhood]\label{def:stable}
Let $W\subseteq J^{r}(S^n, \R^m)$ be singularity type, and let $f\in C^{r+1}(S^n,\R^m)$ with the property that $j^{r}f$ is transversal to $W$. The $W$-stable neighborhood of $f$ in $C^{r+1}(S^n,\R^m)$ is
$$\mathrm{SN}(f, W):=\{g\in C^{r+1}(S^n,\R^m) \mid \textrm{$f$ and $g$ are $W$-transversely homotopic}\}$$
(i.e. there is a homotopy $f_t$ in $C^{r+1}(S^n, \R^m)$ between $f_0=f$ and $f_1=g$ such that for every $t\in [0,1]$ the $r$-th jet extension $j^rf_t$ is transversal to $W$).

\end{definition}
Note the departure from polynomials to general $C^{r+1}$-functions in this definition. In general, it is infeasible to prove bounds on the size of a stable neighborhood. But, if $f$ is given by a polynomial $f=P\mid_{S^n}$, we can measure how non-degenerate its jet with respect to $W$ is. This measure is the distance between $P$ and the so-called discriminant (see also \eqref{def_poly_disc})
\begin{equation}\Sigma_{W,\underline{d}}=\{P\in \mathcal{P}_{n,\underline{d}} \mid \textrm{$j^{r}p$ is not transversal to $W$, where $p=P\mid_{S^n}$}\}.
\end{equation}
The distance measure we take is the \emph{Bombieri-Weyl distance} $\mathrm{dist}_\mathrm{BW}(\cdot, \cdot)$ from Section~\ref{sec:BW}. This distance is particulary tied to Kostlan polynomials. It is interesting to observe that both the approximation result from \cite{Antonio}, as well as our Theorem \ref{thm:main},  are very special of the Kostlan distribution. The reason for this is the behavior of this distribution under the projection onto the spaces of harmonic polynomials. In fact, the results from \cite{Antonio} are related to the more general problem: the estimation of the probability for the projection of $P$ to the subspace of {harmonic polynomials} of degree $L$ to be stable in the sense of Definition~\ref{def:stable}. In the process of proving Theorem~\ref{thm:main} we will also face this more general problem. The following theorem, which is our second main result, is the central piece in this process.

\begin{theorem}[Quantitative stability]\label{thm:quantstab}
There exists a constant $c_1>0$ (depending on $W$) such that, given $P\notin \Sigma_{W}$, and writing $p=P\mid_{S^n}$, if $d_1, \ldots, d_m\geq r+1$ then we have
$$
\{f
\in C^{r+1} (S^n, \mathbb{R}^m)\mid \|f-p\|_{C^{r+1}}<c_1\,\mathrm{dist}_\mathrm{BW}(P, \Sigma_{W,\underline{d}})\} \subseteq \mathrm{SN}(p,W).
$$
\end{theorem}
To appreciate the subtlety of the this result, we remark that the $f$ in the statement is not a polynomial, but rather can be any $C^{r+1}$-function. The space of polynomials of a given degree is finite dimensional, and as all norms on finite dimensional spaces are equivalent, this implies the existence of a constant $c>0$ for which the above statement holds for all \emph{polynomials} $f$; the crucial point here is that the estimate can be made uniform over the whole infinite-dimensional space of $C^{r+1}$ maps. And that the bound only depends on the distance of $p$ to the discriminant $\Sigma_{W,\underline{d}}$ \emph{within the space of polynomials}.

\section{Jet Spaces and Discriminants}\label{sec:jet}
In this section, we briefly introduce some notations and background facts on jet spaces and discriminants. We refer the reader to the textbooks \cite{Hirsch, Arnold} for more details and generalizations.

\subsection{Jet spaces}\label{sec:jet_space}

We recall now the definition of jet manifolds, following \cite{Hirsch}.
Given two smooth manifolds $N$ and $M$, an $r$-jet from $N$ to $M$ is an equivalence class of triples $(x,f,U)$ where $U\subset N$ is an open set, $x\in U$ and $f:U\to M$ is a $C^r$ map; the equivalence relation is: two pairs $(x, f, U)$ and $(y, g, V)$ are equivalent if and only if $x=y$ and in any pair of charts adapted\footnote{Given $f:N\to M$ and $x\in N$, we say that two charts $(C_x, \psi)$ (a chart on a neighborhood $C_x$ of $x$) and
$(B_{f(x)}, \varphi)$ (a chart on a neighborhood $B_{f(x)}$ of $f(x)$) are adapted to $f$ around $x$ if $f(C_x)\subset B_{f(x)}$.} to $f$ around $x$ the maps $f$ and $g$ have the same derivatives up to order $r$. The equivalence class of the triple $(x,f,U)$ is denoted by $j^rf(x)$ and called \emph{the $r$-jet of $f$ at $x$}; the point $x$ is called the \emph{source} of the jet and $f(x)$ the \emph{target}. The set of all $r$-jets from $N$ to $M$ is denoted by $J^r(N,M)$ and the set of all jets with source $x$ is denoted by $J_x^{r}(N, M).$

The most important cases in this paper are $N=\R^{n+1},M=\R^m$ and $N=S^n,M=\R^m$. The next example shows how think of the former.

\begin{example}\label{examp}
Let $N=\R^{n+1}$ and $M=\R^m$. We can represent an $r$-jet of $f:U\to \R^m$ at a point $x$ by the list of derivatives of $f$ at $x$. Therefore, the jet space $J^{r}(\mathbb{R}^{n+1},\mathbb{R}^{m})$ has an explicit manifold structure given by
\beq \label{easy_jet2}
J^{r}(\mathbb{R}^{n+1},\mathbb{R}^{m})\cong \mathbb{R}^{n+1}\times \bigoplus_{j=0}^{r} \R^{mN_j},\quad N_i={n+1+j\choose j},\eeq
such that
$$j^rf(x) = (x, f(x), Df(x),\ldots,D^kf(x),\ldots,  D^rf(x)),$$
where $D^kf(x)\in \R^{mN_j}$ is the \emph{tensor} of the order-$k$ partial derivatives of $f$ at $x$. That is, $J^{r}(\mathbb{R}^{n+1},\mathbb{R}^{m})$ can be seen as the vector bundle over $ \mathbb{R}^{n+1}$, where at each point $x$ we attach all derivatives of polynomials at $x$ up to degree $r$. Another useful interpretation is seeing the symmetric tensors as polynomials.
\beq \label{easy_jet}
J^{r}(\mathbb{R}^{n+1},\mathbb{R}^{m})\cong \mathbb{R}^{n+1}\times \bigoplus_{i=0}^{r} (\mathcal{P}_{n,i})^{\times m}.\eeq
such that $j^rf(x) = (x, f(x), f^{(1)}_1(x),\ldots, f^{(r)}(x))$, where
$$f^{(i)}(x) = \sum_{\alpha_0+\cdots+\alpha_n=i} \,\left[\left(\frac{\partial}{\partial y_0}\right)^{\alpha_0}\cdots \,\left(\frac{\partial}{\partial y_d}\right)^{\alpha_d}f(y)\right]\Biggm|_{y=x}\,x_0^{\alpha_0}\cdots x_n^{\alpha_n}.$$
\end{example}

The manifold structure on $J^{r}(N, M)$ is defined as follows. Given open charts $(U, \psi)$ and $(W, \varphi)$ on $N$ and $M$ respectively we have the bijection
\beq \theta:J^r(U, V)\to J^r(\psi(U), \varphi(V)), \quad j^rf(x)\mapsto j^r\left(\varphi \circ f\circ \psi^{-1}\right)(\psi(x)).\eeq
By \eqref{easy_jet2}, $ J^r(\psi(U), \varphi(V))$ is an open subset of a real vector space. We declare $(J^r(U,V), \theta)$ to be a chart on $J^r(N, M)$ and the set of all such charts gives an atlas, hence a differentiable structure, on $J^{r}(N, M)$. The  map $\theta$ gives local coordinates for the $r$-jet of $f$.

When $N$ and $M$ are real algebraic manifolds, the manifold charts on the jet space are real algebraic as well and the jet space $J^{r}(N, M)$ is also a real algebraic manifold; as a consequence we can define semialgebraic subsets therein, see \cite[Remark 3.2.15]{BCR:98}: $W\subset J^{r}(N, M)$ is semialgebraic if and only if $\theta(W\cap J^{r}(U, V))$ is semialgebraic for every chart $(J^r(U,V), \theta)$.

\begin{definition}\label{r_jet_prolongation}
Let $f:N\to M$, Its $r$-jet prolongation $j^rf:N\to J^r(N, M)$ is $j^rf:x\mapsto j^rf(x).$
\end{definition}
If $N$ and $M$ are smooth, the jet prolongation is smooth, see \cite[Chapter 2.4]{Hirsch}.

Now, we discuss how to think of the jet space $J(X,\R^m)$, where $X\hookrightarrow \R^{n+1}$ is a submanifold. Although the case $X=S^n$ is of main interest to us, it is illustrative to consider a general submanifold. First, we consider the subset of $J(\R^{n+1},\R^m)$ where the base points are points in $X$:
$$J^{r}(\mathbb{R}^{n+1},\mathbb{R}^{m})\vert_{X}\cong X\times \bigoplus_{j=0}^{r} \R^{mN_j}.$$
Then, we have the commutative diagram
\begin{equation}\label{diagram_rho}
\begin{tikzcd}
 J^{r}(\mathbb{R}^{n+1},\mathbb{R}^{m})\vert_{X} \arrow[rdd ] \arrow[rr,"\rho"] &  & {J^{r}(X,\mathbb{R}^{m})}\arrow[ldd] \\
  &  &  \\
  & X  &
 \end{tikzcd}
 \end{equation}
where $\rho$ projects the list of derivatives of a function $f$ at $x\in X$ to the list of derivatives restricted to $T_x X$. In the coordinates from \eqref{easy_jet} this is operator takes the following form. Let $(x,f(x),f^{(1)},\ldots,f^{(r)})\in\mathbb{R}^{n+1}\times \bigoplus_{i=0}^{r} (\mathcal{P}_{n,i})^{\times m}$ be a point representing the $r$-jet of a function $f$. Then,
\begin{equation}\label{rho_explicit}
\rho(x,f(x),f^{(1)},\ldots,f^{(r)}) = (x,f(x),f^{(1)}|_{T_xX},\ldots,f^{(r)}|_{T_xX}).
\end{equation}
That is, $\rho$ restricts the polynomial functions $f^{(i)}$ to  $T_xX$.

\subsection{The $W$-discriminant}\label{sec:W_discr}

Recall from Definition \ref{def:st} that we call $W\subseteq J^{r}(S^n, \R^m)$ a singularity type, if it is semialgebraic and invariant under orthogonal change of variables.  Recall also the notion of transversality to $W$ from Definition \ref{def:transversality}.

Now, we are ready to introduce the $W$-discriminant. It is important to realize that on the one hand $W$ is defined to be a subset of the $r$-th jet space, while on the other hand the associated discriminant lives in the space of $C^{r+1}$ functions!
\begin{definition}[The $W$-discriminant]
Let $ W\subseteq J^{r}(S^n,\mathbb{R}^m) $ be a singularity type.
$$\Sigma_W:= \lbrace f\in C^{r+1}(S^n,\mathbb{R}^m) \mid  j^{r}f  \text{ is not transversal to } W \rbrace \subset C^{r+1}(S^n,\mathbb{R}^m)$$
is called the $W$-discriminant.
\end{definition}
The $W$-discriminant for polynomial systems with degree pattern $\underline{d}$ is defined as
\beq \label{def_poly_disc}
\Sigma_{W,\underline{d}}:= \Sigma_W\cap \mathcal{P}_{n,\underline{d}}.
\eeq
It follows from the definition that
\beq\label{eq:decomposition}
{\Sigma}_{W}=\bigcup_{x\in S^n}{\Sigma}_{W}(x), \quad\text{and}\quad {\Sigma}_{W,\underline{d}}=\bigcup_{x\in S^n}{\Sigma}_{W,\underline{d}}(x).
\eeq
where $ {\Sigma}_{W}(x) $ is the set of $C^{r+1}$ functions~$f$ whose associated map $ j^{r}f$ is not transversal to $ W $ at the point $ x\in S^n $, and $\Sigma_{W,\underline{d}}(x):= \Sigma_W(x)\cap \mathcal{P}_{n,\underline{d}}$.

The discriminant $\Sigma_{W,\underline{d}}$ has the structure of a semialgebraic set with $\mathrm{codim} \Sigma_W \geq 1$.  Moreover, considering the natural induced action of the Orthogonal group $G:=O(n+1)$ on $J^{r}(S^n,\mathbb{R}^m)$ and $\Sigma_{W}$, since $ W $ is $G-$invariant, for every $ g\in G $ we have
\beq\label{eq:transformation}
g\cdot {\Sigma}_{W}(g^{-1}x)\cong {\Sigma}_{W}(x) \quad \text{and}\quad g\cdot {\Sigma}_{W,\underline{d}}(g^{-1}x)\cong {\Sigma}_{W,\underline{d}}(x).
\eeq

\subsection{The degree of the $W$-discriminant}
The next Lemma estimates the degree of the $W$-discriminant $\Sigma_{W,\underline{d}}\subset \mathcal{P}_{n, \underline{d}}$ as a function of $d=\max d_i$. The estimate is a polynomial in $d$. For example, when $\Sigma\subset \mathcal{P}_{n,d}$ is the discriminant for a polynomial having degenerate zero set, then its degree is $(n+1)(d-1)^{n}=O(d^{n})$. Later we will use this estimate on the degree for bounding the probability of a system of Kostlan polynomials to be close to the discriminant $\Sigma_{W,\underline{d}}$.
\begin{proposition}[Degree bound]\label{lemma:discdegree}
Let $W\subseteq J^{r}(S^n, \R^m)$ be a semialgebraic set  and $\underline{d}=(d_1,\ldots,d_m)$ with $d=\max d_i$. There exists a constant $u>0$, which depends on $W$,  and a nonzero polynomial $Q:\mathcal{P}_{n,\underline{d}}\to \R$ of degree bounded by $u\cdot  d^{u}$ such that $\Sigma_{W, \underline{d}}\subset Z(Q)$, where $Z(Q)$ is the zero set of $Q$.
\end{proposition}
\begin{proof}
Fix a point $x_0\in S^n$. We first show that $\Sigma_{W, \underline{d}}(x_0)$ is a semialgebraic subset of $\mathcal{P}_{n,\underline{d}}$ defined by polynomials of degree bounded by some constant $\alpha_1>0$ depending on~$W$ only.

Recall that $\Sigma_{W}(x_0)\subset C^{r+1}(S^n,\mathbb R^m)$ is the set of $C^{r+1}$ functions whose $r$-th jet is not transversal to $W$ at the point $x_0$:
\beq \Sigma_{W}(x_0)=\{f\in C^{r+1}(S^n,\mathbb R^m) \mid j^rf(x_0)\in W \text{ and } \textrm{im}(d_{x_0}j^rf)+T_{j^rf(x_0)}W\neq T_{j^rf(x_0)}J^{r}(S^n, \R^m)\}.\eeq
Since $W$ is defined by semialgebraic conditions, we see that $\Sigma_W(x_0)$ is given by a semialgebraic condition on the list of the first $r+1$ derivatives (i.e. the $(r+1)$-jet) of $f:S^n\to \R^m$ at $x_0$ (the derivative of $j^rf$ involves $j^{r+1}f(x_0)$):
\begin{equation}\Sigma_W(x_0)=\left\{f\in C^{r+1}(S^n, \R^m)\,\bigg|\,  \bigcup_{i=1}^a\bigcap_{j=1}^{b_{i}} \widehat{s}_{i,j}(j^{r+1}f(x_0))*_{i,j}0\right\},
\end{equation}
where $a, b$ are some constants, $\widehat{s}_{i,j}$ are polyomials and $*_{i,j}\in \{<, =, >\}$; all these data depend on $W$ only.
Now, the map $\widehat{j}:\mathcal{P}_{n,\underline{d}}\to J_{x_0}^{r+1}(S^n, \R^m)$ associating to every polynomial map its $(r+1)$-jet at $x_0$ is linear. Moreover, $\Sigma_{W, \underline{d}}(x_0)=\Sigma_{W} (x_0)\cap \mathcal{P}_{n, \underline{d}}$. Taking $s_{i,j}:=\widehat{s}_{i,j}\circ \widehat{j}$
we can therefore write:
 \beq \Sigma_{W,\underline{d}}(x_0)=\bigcup_{i=1}^a\bigcap_{j=1}^{b_{i}} \{s_{i,j}*_{i,j}0\}.\eeq
We have $\mathrm{deg}(s_{i,j}) = \mathrm{deg}(\widehat{s}_{i,j})$, because $\widehat{j}$ is linear. This implies that the degrees of the polynomials $s_{i,j}$ are bounded by some constant $\alpha_1>0$ which only depends on $W$ (but not on $d$).

We now proceed with proving the claim of the proposition.
Let $I$ be the set of all the pairs $(i,j)$ such that $*_{i,j}$ equals ``$=$'' and consider the algebraic set
 \beq Z_{W,\underline{d}}(x_0)=\bigcup_{(i,j)\in I}\{s_{i,j}=0\} \subset \mathcal{P}_{n, \underline{d}}.\eeq
By construction, $\Sigma_{W,\underline{d}}(x_0)\subset Z_{W,\underline{d}}(x_0)$.
Since the $s_{i,j}$ are defined in terms of $W$ only, the cardinality of $I$ is bounded by some constant $\ell$ depending only on $W$.
By \eqref{eq:decomposition} we have $\Sigma_{W, \underline{d}} =\bigcup_{x\in S^n} \Sigma_{W, \underline{d}}(x)$ and therefore, by \eqref{eq:transformation}, we have
\beq \Sigma_{W,\underline{d}} =\bigcup_{g\in O(n+1)}g \cdot \Sigma_{W, \underline{d}}(x_0) \subseteq \bigcup_{g\in O(n+1)}g\cdot Z_{W, \underline{d}}(x_0).\eeq
Let us denote $Z:=\bigcup_{g\in O(n+1)}g\cdot Z_{W, \underline{d}}(x_0)$.
To finish the proof it is enough to find a polynomial $Q$ that vanishes on $Z$ and to bound its degree.

Let $G=\mathrm{GL}(\R^{n+1})$ be the general linear group, and $\rho:G\to \mathrm{GL}(\mathcal{P}_{n,\underline{d}})$ be its representation in the space of polynomial maps  given by change of variables. That is, $\rho(g)P(x) = P(g(x))$. The representation $\rho$ extends to a map between spaces of endomorphisms $\rho:\R^{(n+1)\times (n+1)}\to\R^{N\times N}$, simply by declaring $\rho(g)$ to be the linear map that to a polynomial $P$ associates the new polynomial $\rho(g)\cdot P=P(gx)$. We denote by $h_{i,j}:\R^{(n+1)\times (n+1)}\times\R^{N}\to \R$ the polynomial defined by
\beq h_{i,j}(g, P)=s_{i,j}(\rho(g)\cdot P),\eeq
and we define the incidence set
\beq {\hat{Z}}=\{(g, P)\in \R^{(n+1)\times (n+1)}\times \mathcal{P}_{n,\underline{d}} \mid \forall (i,j)\in I: h_{i,j}(g, P)=0 \}.\eeq
Since the components of the representation $\rho$ have degree at most $d$, and the degrees of the polynomials $s_{i,j}$ are all bounded by $\alpha_1$,
the degree of each $h_{i,j}$ is bounded by $\alpha_1 (d+1)\leq\alpha_2 d $, for some $\alpha_2>0$ (depending on $W$ only and not on $d$). Therefore ${\hat{Z}}$ is defined by at most $\ell$ equations of degree bounded by $\alpha_2 d$.

In order to produce our polynomial $Q$, we move first to the complex numbers and consider the algebraic set ${\hat{Z}}^\C$ defined by the same equations of ${\hat{Z}}$:
\begin{equation}\label{incidence_complex} {\hat{Z}}^\C=\{(g, P)\in \C^{(n+1)\times (n+1)}\times \mathcal{P}^\C_{n,\underline{d}} \mid \forall (i,j)\in I: h_{i,j}(g, P)=0 \}.\end{equation}
Here, $\mathcal{P}^\C_{n,\underline{d}}$ denotes the space of \emph{complex} polynomial systems with degree pattern $\underline{d}$.
Denoting by $\pi:\C^{(n+1)\times (n+1)}\times  \mathcal{P}^\C_{n,\underline{d}} \to  \mathcal{P}^\C_{n,\underline{d}} $ the projection on the second factor, note that
\beq \label{eq:int}Z\subseteq\pi({\hat{Z}}^\C)\cap  \mathcal{P}_{n,\underline{d}} .\eeq
Therefore, in order to get a polynomial vanishing on $Z$, we can sufficiently find a real polynomial~$Q$ vanishing on $\pi({\hat{Z}}^\C)$.

Write the algebraic set ${\hat{Z}}^\C$ as:
\beq\label{eq:irreducible} {\hat{Z}}^\C=\bigcup_{k\geq 0}{\hat{Z}}^\C_k\eeq
where each ${\hat{Z}}_k^\C$ is the union of all irreducible components of ${\hat{Z}}^\C$ of dimension $k$, namely
\beq {\hat{Z}}^\C_k=\bigcup_{i=1}^{\gamma_k} {\hat{Z}}^\C_{k, i}\quad \textrm{with ${\hat{Z}}^\C_{k, i}$ irreducible}.\eeq
Observe that $\C^{(n+1)\times (n+1)}\times  \mathcal{P}^\C_{n,\underline{d}} $ is irreducible and therefore, by \cite[Corollary 2, p.\ 75]{Shafarevich}, the dimension of each component of ${\hat{Z}}^\C$ is bounded below by $(n+1)^2+N-\ell$, where $N=\dim_\C \mathcal{P}^\C_{n,\underline{d}}$. Therefore the previous union \eqref{eq:irreducible} can be written as:
\beq {\hat{Z}}^\C=\bigcup_{k\geq (n+1)^2+N-\ell}{\hat{Z}}^\C_k.\eeq
Fix a number $(n+1)^2+N-\ell\leq k\leq (n+1)^2+N$ (the number $k$ is in range of the possible dimensions for the components of ${\hat{Z}}_{k}^\C$) and observe that
\beq \sum_{i=1}^{\gamma_k} \deg({\hat{Z}}^\C_{k, i})=\deg\left({\hat{Z}}^\mathbb{C}_k\right)\leq (\alpha_2 d)^\ell.\eeq
The reason for this is that, when defining the degree of ${\hat{Z}}^\C_{k}$ we need to intersect it with a generic linear space $L$ of dimension $t=(n+1)^2+N-k$ and this dimension is bounded above by $\ell$, because $k\geq (n+1)^2+N-\ell$. Therefore $L\cap {\hat{Z}}^\C_{k}$ consists of finitely many points in $L\simeq \C^t$, and these points are defined by at most $\ell$ equations. Each of those equations has degree at most $\alpha_2 d$, because the defining equations in  (\ref{incidence_complex}) have degree at most $\alpha_2 d$.
Consequently the number of such points is bounded by $(\alpha_2 d)^t\leq (\alpha_2 d)^\ell.$

We use now \cite[Lemma 2]{Heintz}: since each ${\hat{Z}}^\C_{k,i}$ is irreducible and the projection $\pi$ is linear we have
$\deg(\pi({\hat{Z}}^\C_{k,i}))\leq\deg( {\hat{Z}}^\C_{k,i}).$
In particular this implies:
\beq \deg(\pi({\hat{Z}}_k^\C))=\sum_{i=1}^{\gamma_k}\deg(\pi({\hat{Z}}^\C_{k,i}))\leq\sum_{i=1}^{\gamma_k}\deg({\hat{Z}}^\C_{k,i})\leq (\alpha_2 d)^\ell. \eeq
Let us denote by $\delta_{i,k}$ the degree of $\pi({\hat{Z}}^\C_{k,i})$; since ${\hat{Z}}^\C_{k,i}$ is irreducible, then $\pi({\hat{Z}}^\C_{k,i})$ is irreducible as well and we can apply \cite[Proposition 3]{Heintz} to find a polynomial $F_{k,i}$ of degree bounded by $\delta_{k,i}$ vanishing on $\pi({\hat{Z}}^\C_{k,i}).$ Set now
\beq F_k:=F_{k,1}\cdots F_{k, \gamma_k}.\eeq
The polynomial $F_k$ vanishes on $\pi({\hat{Z}}^\C_k)$ and has degree bounded by
\beq \deg(F_k)=\sum_{i=1}^{\gamma_k}\deg(F_{k,i})\leq \sum_{i=1}^{\gamma_k}\delta_{k,i}\leq (\alpha_2 d)^\ell.\eeq
Define the polynomial
\beq \label{eq:factors}F:=\prod_{k\geq (n+1)^2+N-\ell}F_k\eeq
and observe that, because there are at most $\ell$ factors in \eqref{eq:factors}, then the degree of $F$ is bounded by $\ell(\alpha_2 d)^\ell.$ The polynomial $F$ is not yet what we want, because it might not be real. To fix this, write $F=\mathfrak{Re}(F)+i\mathfrak{Im}(F)$ and define
\beq Q:=\left(\mathfrak{Re}(F)\right)^2+\left(\mathfrak{Im}(F)\right)^2.\eeq
Then $Q$ vanishes on $Z^\C$ (therefore on $Z$ and on $\Sigma\subset Z$) and its degree is bounded by
\beq \deg(Q)\leq 2\ell(\alpha_2 d)^\ell\leq ud^u\eeq
for some constant $u$ which depends on $W$ only.
This finishes the proof of the proposition.
\end{proof}
%%%%%%%%%%%%%%%%%%%%%%%%%%%%%%%%%%%%%%%%%%%%%%%%%%%%%%%%%%%%
%%%%%%%%%%%%%%%%%%%%%%%%%%%%%%%%%%%%%%%%%%%%%%%%%%%%%%%%%%%%
%%%%%%%%%%%%%%%%%%%%%%%%%%%%%%%%%%%%%%%%%%%%%%%%%%%%%%%%%%%%
%%%%%%%%%%%%%%%%%%%%%%%%%%%%%%%%%%%%%%%%%%%%%%%%%%%%%%%%%%%%
%%%%%%%%%%%%%%%%%%%%%%%%%%%%%%%%%%%%%%%%%%%%%%%%%%%%%%%%%%%%
%%%%%%%%%%%%%%%%%%%%%%%%%%%%%%%%%%%%%%%%%%%%%%%%%%%%%%%%%%%%
%%%%%%%%%%%%%%%%%%%%%%%%%%%%%%%%%%%%%%%%%%%%%%%%%%%%%%%%%%%%
%%%%%%%%%%%%%%%%%%%%%%%%%%%%%%%%%%%%%%%%%%%%%%%%%%%%%%%%%%%%
\section{Norms and polynomials}\label{sec:norms_and_polys}
In this section, we first introduce the decomposition of the space of homogeneous polynomials into the so-called harmonic basis. Then we define several norms on the space of polynomials, which will be used in the proofs later.
\subsection{Harmonic polynomials}\label{sec:harmonics}
The switch from the monomial basis to the harmonic basis will be the key to obtain a low degree approximation of singular loci.
\begin{definition}
Let $0\leq \ell \leq d$ the space of homogeneous \emph{harmonic} polynomials is
$$\mathcal{H}_{n,\ell}:= \Big\{ P\in \mathcal{P}_{n,\ell} \mid \sum_{i=0}^n\frac{\partial^2}{\partial x_i^2}P=0\Big\}.$$
\end{definition}

The space $\mathcal{P}_{n,d}$ can be decomposed as:
\begin{equation}\label{eq:decomp} \mathcal{P}_{n,d}=\bigoplus_{d-\ell\in 2\mathbb{N}}\|x\|^{d-\ell}\mathcal{H}_{n,\ell}.\end{equation}
The decomposition \eqref{eq:decomp} has two important properties (see \cite{Kostlan95}):
\begin{itemize}
\item [(i)]Given a scalar product which is invariant under the action of $O(n+1)$ on $\mathcal{P}_{n,d}$ by change of variables, the decomposition \eqref{eq:decomp} is orthogonal for this scalar product.
\item [(ii)] The action of $O(n+1)$ on $\mathcal{P}_{n,d}$ preserves each $\mathcal{H}_{n,\ell}$ and the induced representation on the space of harmonic polynomials is irreducible. In particular, there exists a unique, up to multiples, scalar product on $\mathcal{H}_{n,\ell}$ which is $O(n+1)$-invariant.
\end{itemize}

\subsection{Norms of polynomials}\label{sec:norms}
We continue this section by defining several norms on the space of polynomials $\mathcal{P}_{n,d}$ realized as a subspace of $C^r(S^n,\R)$. In general, we give $\mathcal{P}_{n,\underline{d}}$ the structure of a normed space by endowing each space $\mathcal{P}_{n, d_i}$ with a norm $\| \cdot \|_{d_i}$. The norm on $\mathcal{P}_{n,\underline{d}}$ is:
$$\| P \|:=\left(\sum_{i=1}^{m} \| P_i \|_{d_i}^{2}\right)^{1/2} \text{for}\quad P=(P_1,\ldots,P_m)\in \mathcal{P}_{n,\underline{d}}$$\\
We identify $\mathcal{P}_{n,d}$ with its image in $C^r(S^n,\R)$ given by
$\mathcal{S}_{n,d}:=\lbrace p:S^n\to \R \mid p=P|_{S^n}, P\in  \mathcal{P}_{n,d}\rbrace.$
The decomposition \eqref{eq:decomp} induces a decomposition:
\begin{equation} \label{eq:decomp2}\mathcal{S}_{n,d}=\bigoplus_{d-\ell\in 2\mathbb{N}}V_{n,\ell}\quad \text{with}\quad V_{n,\ell}=\mathcal{H}_{n,\ell}|_{S^n}.\end{equation}
Writing $P=\sum_{\ell}P_\ell$ with each $P_\ell\in \|x\|^{d-\ell}\mathcal{H}_{n,\ell}$ as in \eqref{eq:decomp}, when taking restrictions to the unit sphere we have $p=\sum_{\ell}p_\ell$ with each $p_\ell$ the restriction to $S^n$ of a polynomial of degree $\ell$: in other words, the restriction to the unit sphere ``does not see'' the $\|x\|^{d-\ell}$ factor, which is constant on the unit sphere.

Here follows the definition of some relevant norms that we will use in this paper. The first three of them are induced by an orthogonally invariant scalar product: by property (i) above the decomposition \eqref{eq:decomp2} is orthogonal for all of them.

\subsubsection{The Bombieri-Weyl norm} \label{sec:BW}  Let $P=\sum_{\alpha_0+\cdots+\alpha_n=d} a_{\alpha_0,\ldots,\alpha_0} \, \sqrt{\tfrac{d!}{\alpha_0!\cdots\alpha_n!}}\,x_{0}^{\alpha_0}\cdots x_n^{\alpha_n}$ be homogeneous polynomial of degree $d$. The Bombieri-Weyl norm of $P$ is defined by
$$ \| P\|_{\mathrm{BW}}^2:= \sum_{\alpha_0+\cdots+\alpha_n=d} \,(a_{\alpha_0,\ldots,\alpha_0} )^2.$$
Comparing with \eqref{Kostlan}, we see that Kostlan polynomials are given by a multivariate standard Gaussian distribution with respect to the Bombieri-Weyl product. This is the ``close connection'' we have mentioned earlier in the paper. The Bombieri-Weyl distance is $d_\mathrm{BW}(P,Q) = \Vert P-Q\Vert.$
\subsubsection{The $L^2$-norm} This norm is the $L^2$-norm of $p$ defined by
\begin{equation} \|P\|_{L^2}^2:=\int_{S^n} p(x)^2\,\mathrm{d}x,\end{equation}
where ``$\mathrm{d}x$'' denotes integration with respect to the standard volume form of the sphere.

\subsubsection{The Sobolev $q$-norm}
Let $P=\sum_{\ell}P_\ell$ be the decomposition into the harmonic polynomials basis; see \eqref{eq:decomp}. Then the Sobolev $q$-norm is defined by
\begin{equation} \|P\|_{H^q}^2:= \|P_0\|_{L^2}^2+\sum_{d-\ell\in2\mathbb{N}}\ell^{2q}\|P_\ell\|_{L^2}^2.\end{equation}
Note that  $\|P_0\|_{L^2}^2=0$ when $d$ is odd. Moreover $\|\cdot\|_{H^0}=\|\cdot\|_{L^2}$.\\

\subsubsection{$C^r$-norm}\label{section:norm}
The $C^r$ norm is defined for all $C^r$ functions. The $C^r$ norm for polynomials is then just the restriction to the space of polynomials. We give the general definition.

We fix an orthogonal invariant norm $\nu$ on the vector space $\bigoplus_{j=0}^r\R^{mN_j}$. Moreover, we let $\pi: J^r(\R^{n+1},\R^m)\to \bigoplus_{j=0}^r\R^{mN_j}$ be the projection that
removes the base point. Then, we define the norm of a jet to be $\nu(z):=\nu(\pi(z))$; i.e., the norm of a jet is the norm of the point in the fiber.
For a given $\eta\in J^{r}_x(S^n, \R^m) $ we define:
\beq\label{eq:hat} \hat{\nu}(\eta)=\inf_{z\in J^r_x(\R^{n+1},\R^m):\, \rho(z)=\eta} \nu (z),\eeq
where $\rho$ is the restriction map from \eqref{diagram_rho},
and for a function $f\in C^r(S^n,\R^m) $ we then set
\begin{equation}\label{def_C_r_norm}
 \|f\|_{\CC^r}^{\nu}:=\sup_{x\in S^n}\, \hat{\nu}(j^rf(x)).
 \end{equation}

\begin{remark}
The definition of the $C^r$ norm includes the choice of a norm $\nu$. Yet, the topology it induces is independent of this choice. This is because $\nu$ is a norm on a finite dimensional real vector space, and all norms on finite dimensional real vector spaces are equivalent; we call this topology the $C^r$ topology. Here is another way for obtaining it.
Consider the natural map
$j^r:C^r(S^n,\R^m)\to C^0(S^n,J^{r}(S^n,\mathbb{R}^m)).$
Because the sphere $S^n$ is compact, the strong and the weak topology on $C^0(S^n,J^{r}(S^n,\mathbb{R}^m))$ coincide and by \cite[Chapter 2, Theorem 4.3]{Hirsch} the image of this map is closed in the strong topology. In particular we can immediately induce a topology on the space $C^r(S^n,\R^m)$, which is the $C^r$ topology, \cite[pag. 62, before Theorem 4.4]{Hirsch}.
\end{remark}

\subsection{An inequality between the $C^r$ norm and the Sobolev norm}

In the last part of this section we want to prove an inequality between the $C^r$ norm and the Sobolev norm. This inequality will be useful for the proofs in the next section.
We note that by endowing $\mathcal{P}_{n,\underline{d}}$ with the product norm, in order to compare norms on $\mathcal{P}_{n,\underline{d}}$, we can sufficiently reduce to comparison of the corresponding norms for a single polynomial $P\in \mathcal{P}_{n,d}$ rather than a vector of polynomials.

We first recall the following result from \cite{Seeley}:
\begin{theorem}\label{seeley}
Let $\alpha=(\alpha_0,\ldots,\alpha_n)$ be a list of nonnegative integers and $ \partial^{\alpha}=\partial^{\alpha_{0}}_{x_0}\ldots \partial^{\alpha_{0}}_{x_0}$ be the associated differential operator.
There are constants $\beta_{(\alpha,n)} $ that only depend on $\alpha$ and $n$ such that
$$ \int_{S^n}|\partial^{\alpha}P(x)|^{2} \mathrm{d}x\leq \beta_{(\alpha,n)}\ell^{2|\alpha|}\int_{S^n}|P(x)|^{2} \mathrm{d}x \; \text{ for every } P\in \mathcal{H}_{n,\ell},$$
where $\vert \alpha\vert=\alpha_{0}+\ldots+\alpha_{n}$.
\end{theorem}

The following Proposition connects the $C^r$-norm with the Sobolev $q$-norm.
\begin{proposition}\label{prop1}
Let $\nu$ be the norm defining the $C^r$ norm in equation \eqref{def_C_r_norm}.
There exists a constant $c=c(r,n,\nu)>0$ depending on $r,n$ and $\nu$ such that, if $q\geq r+\frac{n-1}{2}$, we have $\|P\|_{\CC^r}^\nu\leq c\, \sqrt{d} \,\|p\|_{H^{q}}.$
\end{proposition}
\begin{proof}
Let $ p=P\vert_{S^n} $ and $ p=\sum_{d-\ell\in 2\N}p_\ell $ be the decomposition of $p$ in the harmonic basis from~\eqref{eq:decomp2}. By definition  \eqref{def_C_r_norm} of the $C^r$-norm, we have
$
\|p_\ell \|_{C^r}\leq \sup_{x\in S^n} \nu (j^{r}P_{\ell}(x)).
$
Moreover, there exists a constant  $c_1=c_1(r,n,\nu)$ such that
$\nu (j^{r}P_{\ell}(x))\leq c_1 \, \sup_{x\in S^n} ( \sum_{\vert \alpha\vert \leq r}\vert \partial^{\alpha}P_\ell(x)\vert),$
where $ \partial^{\alpha}=\partial^{\alpha_{0}}_{x_0}\ldots \partial^{\alpha_{0}}_{x_0}$ and $\vert \alpha\vert=\alpha_{0}+\ldots+\alpha_{n}  $. This is because the right-hand side of the equation is a multiple of the $L^1$-norm on the fibers of the jet space, and all norms on a finite dimensional vector space are equivalent. Summarizing, we have
\beq\label{ineq1}
\|p_\ell \|_{C^r}^\nu\leq c_1 \, \sup_{x\in S^n} \Big( \sum_{\vert \alpha\vert \leq r}\vert \partial^{\alpha}P_\ell(x)\vert\Big).
\eeq

Recall from \eqref{eq:decomp2} the definition of $V_{n,\ell}=\mathcal{H}_{n,\ell}|_{S^n}$.
For every $\ell=0, \ldots, d$ the space $V_{n,\ell}$ with the $L^2$-scalar product is a reproducing kernel Hilbert space, that is there exists $Z_\ell:S^n\times S^n\to \R$ such that for every $h_\ell\in V_{n,\ell}$:
\begin{equation}\label{eq:reproducing} h_\ell(x)=\int_{S^n}h_\ell(\theta)Z_{\ell}(x, \theta)\mathrm{d}\theta.\end{equation}
The function $Z_\ell$ (the ``zonal harmonic'') is defined as follows: letting $\{y_{\ell, j}\}_{j\in J_{\ell}}$ be an $L^2$-orthonormal basis for $V_{n, \ell}$ we set
$Z_\ell(\theta_1, \theta_2)=\sum_{j\in J_{\ell}}y_{\ell, j}(\theta_1)y_{\ell, j}(\theta_2)$
(written in this way \eqref{eq:reproducing} is easily verified, see \cite[Chapter 5]{HFT} and \cite[Proposition 5.27]{HFT} for more properties of the Zonal harmonic). Then from this it follows that:
\begin{equation}\label{eq:zonal} \|Z_\ell(\theta_1, \cdot)\|^2_{L^2}=\langle Z_\ell(\theta_1, \cdot), Z_\ell(\theta_1, \cdot)\rangle_{L^2}=Z_\ell(\theta_1, \theta_1)=\frac{\dim (V_{n,\ell})}{\textrm{vol}(S^n)}=O(\ell^{n-1}),\end{equation}
where the last identity follows from\footnote{The constant ``$\frac{1}{\textrm{vol}(S^n)}$'' appears because in \cite{HFT} the normalized $L^2(S^n)$ space is used, i.e. the convention $\textrm{vol}(S^n)=1$ is adopted.} \cite[Proposition 5.27 (d)]{HFT} and \cite[Proposition 5.8]{HFT}. Writing $p^\alpha_\ell := \partial^{\alpha}P_\ell |_{S^n}$ we obtain the following.
\begin{alignat}{2}\label{ineq2}
| \partial^{\alpha}P_\ell|_{S^n}(x)|&=\left\vert\int_{S^n}p^\alpha_\ell(\theta) \, Z_{\ell}(x, \theta)|\mathrm{d}\theta\right\vert, \quad &&\text{by \eqref{eq:reproducing}}\\
&\leq \|\partial^{\alpha}P_\ell \,\|_{L^2} \; \|Z_\ell(x, \cdot)\|_{L^2},&&\text{by the Cauchy-Schwartz inequality}\\
&\leq c_2(\alpha,n)\, \ell^{|\alpha|+\frac{n-1}{2}} \, \|P_\ell\|_{L^2},&&\text{by Theorem \ref{seeley} and \eqref{eq:zonal}},
\end{alignat}
where $c_2(\alpha,n)$ is a constant that depends on $\alpha$ and $n$.
From the above inequalities it follows that
\begin{alignat}{2}\label{ineq3}
\|P\|_{\CC^r}^\nu&\leq \sum_{d-\ell \in2\mathbb{N}} \|p_\ell\|_{\CC^r}, && \text{by the triangle inequality}\\
& \leq \sum_{d-\ell \in2\mathbb{N}}c_1\sup_{x\in S^n} \Big( \sum_{\vert \alpha\vert \leq r}\vert \partial^{\alpha}P_\ell(x)\vert\Big),\quad &&\text{by \eqref{ineq1}}\\
& \leq c_3\sum_{d-\ell \in2\mathbb{N}}\ell^{r+\frac{n-1}{2}}\|p_{\ell}\|_{L^{2}}, && \text{by \eqref{ineq2}},
\end{alignat}
where $c_3$ is a constant that depends on $r,n,\nu$ (the dependence on $\alpha$ has been moved into the dependence on $r$). Then, we use the Cauchy-Schwartz inequality for $v_1=(\ell^{r+\frac{n-1}{2}}\|p_{\ell}\|_{L^2})_{d-\ell \in2\mathbb{N}}$ and $v_2=(1, \ldots, 1)$
so that:
\begin{align}
\sum_{d-\ell \in 2\mathbb{N}}\ell^{r+\frac{n-1}{2}}\|p_\ell\|_{L^2}&=\langle v_1, v_2\rangle\\
&\leq \|v_1\|\|v_2\|\\
&=\Big(\sum_{d-\ell \in 2\mathbb{N}}\ell^{2r+n-1}\|p_\ell\|_{L^2}^2\Big)^{1/2} \Big(\sum_{d-\ell \in 2\mathbb{N}}1\Big)^{1/2} \\
&\leq \sqrt{d}\,\Big(\sum_{d-\ell \in 2\mathbb{N}}\ell^{2r+n-1}\|p_\ell\|_{L^2}^2\Big)^{1/2}.
\end{align}
Plugging this into \eqref{ineq3} we obtain
$\|P\|_{\CC^r}^\nu \leq c_3\,\sqrt{d}\,\|P\|_{H^{q}}$ for  $q\geq r+\tfrac{n-1}{2}.$
This finishes the proof.
\end{proof}

\section{Proof of the quantitative stability Theorem}\label{sec:proof_thm_2}
In order to prove Theorem \ref{thm:quantstab} we need to recall Thom's isotopy lemma. We give a variant which is uses our notation. Recall from \eqref{def_C_r_norm} the definition of the $C^{r}$ norm $\Vert \cdot\Vert_{C^r}^\nu$. We denote the associated distance function $\mathrm{dist}_{C^{r+1}}^\nu(\cdot, \cdot)$. Furthermore, recall from Definition \ref{def:transversality} that $j^{r}f\pitchfork W$ means that $j^rf$ is transversal to $W$.

\begin{lemma}[Isotopy Lemma]
\label{lm1}
Let $f,g\in C^{r+1}(S^n,\mathbb{R}^{m})$ such that
$$  j^{r}f\pitchfork W \quad\text{and}\quad \| f-g\|_{C^{r+1}}^\nu < \mathrm{dist}_{C^{r+1}}^\nu(f,\Sigma_W).$$
Then $g\in\mathrm{SN}(f,W)$.
\end{lemma}
\begin{proof}The condition stated guarantees that the homotopy $f_t=(1-t)f+tg$ has the property that for every $t\in [0,1]$ the jet $j^rf_t$ is transversal to $W$. Thus we have an induced homotopy of maps between smooth manifolds $j^rf_t:S^n\to J^r(S^n, \R^m)$ which is transversal to all the strata of $W\subseteq J^r(S^n, \R^m)$ for every $t\in [0,1]$. In particular, this holds for $g=f_1$ and the conclusion follows from the definition of Stable Neighborhood in Definition \ref{def:stable}.
\end{proof}

We also need the following helpful lemma.
\begin{lemma}\label{helpful}
Let $P = \sum_{\alpha_0+\cdots+\alpha_n=d}\, c_\alpha\,x^\alpha \in \mathcal{P}_{n,\underline{d}}$ and $p=P|_{S^n}$. Furthermore, let $T:\mathcal{P}_{n,\underline{d}} \to \mathcal{P}_{n,\underline{d}}$ and $Y:\mathcal{P}_{n,\underline{d}} \to \mathcal{P}_{n,\underline{d}}$ be the linear maps defined by
$$T(P) = \sum_{|\alpha_0|\leq {d-r-2}}c_\alpha\, x^{\alpha} \quad\text{and}\quad Y(P) = \sum_{|\alpha_0|\geq {d-r-1}}c_\alpha\,x^{\alpha},$$
and let $\tau(p)$ and $\gamma(p)$ be the restrictions of $T(P)$ and $Y(P)$ to $S^m$ and $e_0=(1,0, \ldots, 0)\in S^n$.
Then,
$j^{r+1}p(e_0) = j^{r+1}\gamma(p)(e_0).$
\end{lemma}
\begin{proof}
Note first that $P= Y(P)+T(P).$ All the derivatives of order up to $r+1$ of $T(P)$ vanish at $e_0$, and consequently the restrictions of the corresponding polynomials to $T_{e_0}S^m$ will also be zero.
\end{proof}

Now we are ready to prove Theorem \ref{thm:quantstab}.
\begin{proof}[Proof of Theorem \ref{thm:quantstab}]
For every $x\in S^n$ let\beq Z_{W}(x):=\{j^{r+1}f(x) \mid \textrm{$j^rf$ is not transversal to $W$ at $x$}\} \subset J^{r+1}_x(S^n, \R^m);
\eeq
that is, $Z_{W}(x) = j^{r+1}(\Sigma_W(x))$.
For $d_1,\ldots,d_m\geq {r+1}$, the map $S^n\times \mathcal{P}_{n, \underline{d}}\to J^{r+1}(S^n, \R^m)$ given by $(x, P)\mapsto j^{r+1}p(x)$, where $p=P|_{S^n}$, is a submersion (see \cite[Section 1.7]{eliashberg}) and therefore $Z_{W}(x)$ can be defined using only polynomial functions:
\beq\label{eq:submersion}
Z_{W}(x)=j^{r+1}(\Sigma_{W,\underline{d}}(x)).\eeq
Now, let $\mathrm{dist}^{\hat{\nu}}$ be the distance function induced by the norm $\hat{\nu}$ on $J_{x}^{r+1}(S^n, \R^m)$ as defined in~\eqref{eq:hat}.
The idea of the proof is to show the two inequalities
\beq\label{eq:12} \inf_{x\in S^n}\mathrm{dist}^{\hat{\nu}}(j^{r+1}p(x), Z_{W}( x))\leq \mathrm{dist}_{C^{r+1}}^\nu(p, \Sigma_W).\eeq
and
\beq\label{eq:11} \mathrm{dist}_\mathrm{BW}(P, \Sigma_{W, \underline{d}})\leq c_1'\inf_{x\in S^n}\mathrm{dist}^{\hat{\nu}}(j^{r+1}p(x), Z_{W} (x)),\eeq
for some $c_1'>0$.
The result will follow from these two inequalities. In fact, setting $c_1=(c_1')^{-1}$ and combining \eqref{eq:11} with \eqref{eq:12} we therefore will have:
\beq \|f-p\|_{C^{r+1}}^\nu< c_1\mathrm{dist}_\mathrm{BW}(P, \Sigma_{W, \underline{d}})\implies  \|f-p\|_{C^{r+1}}^\nu<\mathrm{dist}^\nu_{C^{r+1}}(p, \Sigma_W)\implies f\in \mathrm{SN}(p, W),\eeq
where for the second implication we have used the definition of stable neighborhood (Definition \ref{def:stable}) and Lemma \ref{lm1} above.

Let us prove the two inequalities, starting from \eqref{eq:11}. For every $x\in S^n$ let $g_x\in O(n+1)$ be an orthogonal transformation mapping $x$ to $e_0=(1, 0,\ldots, 0)$ and define $P^x$ to be the polynomial:
\beq P^x:=(P_1\circ g_x, \ldots, P_m\circ g_x), \text{ where } P=(P_1,\ldots,P_m).\eeq
Then, using $\Sigma_{W, \underline{d}}=\bigcup_{x\in S^n}\Sigma_{W, \underline{d}}(x),$ we can write:
\beq \mathrm{dist}_\mathrm{BW}(P, \Sigma_{W, \underline{d}})=\inf_{x\in S^n}\mathrm{dist}_\mathrm{BW}(P, \Sigma_{W, \underline{d}}(x))=\inf_{x\in S^n}\mathrm{dist}_\mathrm{BW}(P^x, \Sigma_{W, \underline{d}}(e_0));\eeq
the last inequality due to the orthogonal invariance of the Bombieri-Weyl norm.
Writing out the definition of the distance we have
\begin{equation}
\mathrm{dist}_\mathrm{BW}(P, \Sigma_{W, \underline{d}})=\inf_{x\in S^n} \inf_{Q\in\Sigma_{W,\underline{d}}(e_0)}\|P^x -Q\|_{\mathrm{BW}}.
\end{equation}
Let $T,Y,\tau,\gamma$ be as in Lemma \ref{helpful}. This lemma implies that, if $Q\in \Sigma_{W, \underline{d}}(e_0)$ is in the discriminant, any other polynomial of the form $\tilde{Q}=Q+T(F)$, for $F\in \mathcal{P}_{n, \underline{d}}$, is also in the discriminant $\Sigma_{W, \underline{d}}(e_0),$ because all the derivatives of $T(F)$ up to order $r+1$ vanish at $e_0$.  In particular, if we want to minimize the quantity $\|P^x -Q\|_{\mathrm{BW}}$ for $Q\in\Sigma_{W,\underline{d}}(e_0)$, we can restrict ourself to the polynomials $Q$ such that: $T(Q)=T(P^x)$ and $Y(Q)\in \Sigma_{W, \underline{d}}(e_0)$ (notice that $Y(P)$ and $T(Q)$ are orthogonal for any pairs of polynmials $P$ and $Q$).
Therefore we get
\begin{equation}\label{111}
\mathrm{dist}_\mathrm{BW}(P, \Sigma_{W, \underline{d}})
=\inf_{x\in S^n}\inf_{Q\in\Sigma_{W,\underline{d}}(e_0)}\left(\sum_{i=1}^m\bigg\|Y(P_i^x)-Y(Q_i)\bigg\|_\mathrm{BW}^2\right)^{1/2}
\end{equation}
For $i=1, \ldots, m$,  let
$$P_i^x(y)=\sum_{\alpha_1+\cdots+\alpha_n=d}\, c_\alpha(P_i^x)\,y^\alpha \quad\text{and}\quad Q(y)=\sum_{\alpha_1+\cdots+\alpha_n=d}\, c_\alpha(Q)\,y^\alpha.$$
be the expansions of $P_i^x=P_i\circ g_x$ and $Q_i$ (the $i$-the entry of $Q$) in the monomial basis.  Then, following \eqref{111} we have
\begin{align*}
\mathrm{dist}_\mathrm{BW}(P, \Sigma_{W, \underline{d}})&=\inf_{x\in S^n}\inf_{Q\in\Sigma_{W,\underline{d}}(e_0)}\left(\sum_{i=1}^m\sum_{|\alpha_0|\geq d-r-1}|c_\alpha(P_i^x)-c_\alpha(Q_i)|^{2}\,\dfrac{\alpha_{0}!\cdots \alpha_{n}!}{d!}\right)^{1/2}\\
& \leq \inf_{x\in S^n}\inf_{Q\in\Sigma_{W, \underline{d}}(e_0)}\left(\sum_{i=1}^m\sum_{|\alpha_1|+\ldots+|\alpha_n|\leq r+1}|c_\alpha(P_i^x)-c_\alpha(Q_i)|^{2} \right)^{1/2}.
\end{align*}
Let $q=Q|_{S^n}$  and $p^x = P^x|_{S^n}$. Observe now that $\left(\sum_{i=1}^m\sum_{|\alpha_1|+\ldots+|\alpha_n|\leq r+1}|c_\alpha(P_i^x)-c_\alpha(Q_i)|^{2}\right)^{1/2}$ is the Frobenius norm of $\pi(j^{r+1}(p^x-q)(e_0))$, where, as before, $\pi$ is the projection on $J_{e_0}^{r+1}(S^n, \R^m)$ that removes the base point. Since all norms on finite dimensional spaces are equivalent, there exists $c_1'$ such that
\begin{align*} \left(\sum_{i=1}^m\sum_{|\alpha_1|+\ldots+|\alpha_n|\leq r+1}|c_\alpha(P_i^x)-c_\alpha(Q_i)|^{2} \right)^{1/2}\leq c_1'\hat{\nu}(j^{r+1}(p^x-q)(e_0)).\end{align*}
In particular, we have
\begin{alignat*}{2}
\mathrm{dist}_\mathrm{BW}(P, \Sigma_{W, \underline{d}})& \leq c_1' \inf_{x\in S^n}\inf_{Q\in\Sigma_{W,\underline{d}}(e_0)}\hat{\nu}_{e_0}(j^{r+1}(p^x-q)(e_0))&\\
& =c_{1}' \inf_{x\in S^n} \mathrm{dist}^{\hat{\nu}}(j^{r+1}p^x(e_0), Z_{W}(e_0))& \text{ by the definition of $Z_W(e_0)$ in \eqref{eq:submersion},}\\
&=c_{1}' \inf_{x\in S^n} \mathrm{dist}^{\hat{\nu}}(j^{r+1}p(x), Z_{W}(x))&\text{ by orthogonal invariance};
\end{alignat*}
This proves \eqref{eq:11}.

For the proof of \eqref{eq:12} we argue as follows. Recalling the definition of $\Sigma_W(x)$ given in \eqref{eq:decomposition} we have
\begin{align}\inf_{x\in S^n}\mathrm{dist}^{\hat{\nu}}(j^{r+1}p(x), Z_{W}(x))&=\inf_{x\in S^n}\inf_{f\in \Sigma_{W}(x)}\hat{\nu}(j^{r+1}(p-f)(x))\\
&\leq \inf_{x\in S^n}\inf_{f\in \Sigma_{W}(x)}\sup_{y\in S^n}\hat{\nu}(j^{r+1}(p-f)(y)),
\end{align}
By definition of the $C^r$ norm \eqref{def_C_r_norm} we have $\sup_{y\in S^n}\hat{\nu}(j^{r+1}(p-f)(y))=\|f-p\|_{C^{r+1}}^\nu$, and therefore
\begin{align}\mathrm{dist}_\mathrm{BW}(P, \Sigma_{W, \underline{d}})& \leq \inf_{x\in S^n}\inf_{f\in \Sigma_{W}(x)}\|f-p\|^{\nu}_{C^{r+1}}\\
&=\inf_{x\in S^n}\mathrm{dist}_{C^{r+1}}^\nu(p, \Sigma_W(x))\\
&=\mathrm{dist}_{C^{r+1}}^\nu(p, \Sigma_W),\end{align}
where in the last line we have used the fact that $\Sigma_W=\bigcup_{x\in S^n}\Sigma_W(x)$. This gives \eqref{eq:12}.
%As before let $p=P|_{S^n}$.
%The idea in the proof is to show that the inequality $\|f-p\|_{C^{r+1}}<c\ \textrm{dist}_\mathrm{BW}(P, \Sigma_{W,\underline{d}})$ implies
%$
%\|f-p\|_{C^{r+1}}\leq \mathrm{dist}_{C^{r+1}}(p, \Sigma_W).
%$
%Then, as a consequence of Lemma \ref{lm1} we have $f\in \mathrm{SN}(p,W)$.
%
%
%Pick now  $f\in \Sigma_W$. By (\ref{def_C_r_norm}) we have
%\beq \mathrm{dist}_{C^{r+1}}(p,f) = \sup_{x\in S^n}\hat{\nu}(j^{r+1}f(x)-j^{r+1}p(x))= \sup_{x\in S^n}\, \inf_{z\in J^{r+1}(\R^{n+1},\R^m) \atop \rho(z)=j^{r+1}(f-p)(x)} \nu (z)
%\eeq
%Since $S^n$ is compact, the supremum is a maximum. Moreover, by orthogonal invariance of $\nu$ we can assume that the maximum is attained at $e_0= (1,0,\ldots,0)\in S^n$.

\end{proof}

\section{Low degree approximation}\label{sec:low_degree_approx}
The goal of this section is proving Theorem \ref{thm:main}. For this we need to estimate the probability for the following projection of a polynomial stay inside the stable neighborhood.

Let us recall from the introduction our definition of the projection operator on polynomial maps of smaller degree. For an integer $L\in \{0, \ldots, d\}$, we set
\begin{equation}
p|_{L}:=\sum_{\ell\leq L,\, d-\ell \in 2\mathbb{N}}p_\ell.
\end{equation}
where $ p=\sum_{d-\ell\in 2\mathbb{N}}p_\ell$ is the harmonic decomposition of (each component of) $p=P|_{S^n}$. We extend this definition to polynomial maps $p=(p_1, \ldots, p_m)$ as done in \eqref{eq:projmap}. Next, we define the event of this projection to stay inside the stable neighborhood.

\begin{definition}[Stability event]\label{distanceevent}Let $W\subseteq J^{r}(S^n, \R^m)$ be a singularity type and  $c_1>0$ be the constant given by the Quantitative Stability Theorem \ref{thm:quantstab}.  For an integer $L\in \{0, \ldots, d\}$, we denote by $E_L$ the event
$
E_L=\{P\in\mathcal P_{n,\underline{d}} \mid \|p-p|_{L}\|_{C^{r+1}}<c_1\,\mathrm{dist}_{\mathrm{BW}}(P, \Sigma_W)\}
$, where, as before, $p=P|_{S^n}$.
\end{definition}
Notice that the Quantitative Stability Theorem \ref{thm:quantstab} implies that $E_L\subseteq A_L$ (the low-degree approximation event defined in Definition \ref{def:stabilityevent}) and in particular a lower bound on the probability of $E_L$  serves also as a lower bound for the probability of $A_L$.
\begin{lemma}
With the above notations, we have $E_L\subseteq A_L$.
\end{lemma}
\begin{proof}
Let $P\in E_L$ and $p=P|_{S^n}$. Theorem \ref{thm:quantstab} implies $p|_L\in \mathrm{SN}(p, W)$. Then $j^rp$ and $j^rp|_L$ are connected by a homotopy transversal to $W$. The conclusion follows from \cite[Th\'eor\`eme~2.D.2]{thom}.
\end{proof}
Note that type-$W$ singular locus of polynomials in the event $A_L$ is ambient isotopic to the type-$W$ singular locus of a polynomial of degree $L$. In other words, polynomials in $E_L$ and in $A_L$ ``look like'' polynomials of lower degree.

For obtaining a bound on the probability of $E_L$ we first need to prove the next proposition.
\begin{proposition}\label{prop:discnhood}Let $W\subseteq J^{r}(S^n, \R^m)$. There exist constants  $a_1,a_2, a_3>0$ which depend on $W$, such that the following holds. For every $s\geq a_1\,d^{a_2}$ we have for systems of Kostlan polynomials $P=(P_1, \ldots, P_m)\in \mathcal{P}_{n, \underline{d}}$, with $\underline{d}=(d_1, \ldots, d_n)$,
\begin{equation} \mathbb{P}\big\{\|P\|_{\mathrm{BW}}\leq s \,\mathrm{dist}_{\mathrm{BW}}(P, \Sigma_W)\big\}\geq 1-\frac{a_3\,d^{a_2}}{s},\end{equation}
where $d=\max d_i$.
\end{proposition}
\begin{proof}The proof of this Proposition is the same as the proof of \cite[Proposition 4]{Antonio}, of which the current statement is just a simple generalization.
Let $Q$ be the polynomial given by  Proposition~\ref{lemma:discdegree}. Then $\Sigma_W$ is contained in $Z(Q)$, the zero set of $Q$, and we can apply \cite[Theorem 21.1]{BuCu} as follows.

Let $N=\dim(\mathcal{P}_{n, \underline{d}})=\sum_{k=1}^m{d_k+n\choose n}\leq  d^{n m}$ and $D=\deg(Q)$.
Denoting by $d_\textrm{sin}$ the sine distance in the sphere, \cite[Theorem 21.1]{BuCu} tells that there exists a constant $C_3>0$ such that for all $s\geq 2DN$ we have:
\beq \frac{\textrm{vol}\left(\left\{\textrm{$p\in S^{N-1}$ such that $\frac{1}{d_\textrm{sin}(p, \overline{\Sigma})}\geq s$}\right\}\right)}{\textrm{vol}(S^{N-1})}\leq C_3D N s^{-1}.\eeq
Taking the cone over the set $\{\textrm{$p\in S^{N-1}$ such that $\frac{1}{d_\textrm{sin}(p, \overline{\Sigma})}\geq s$}\}$, we can rewrite the previous inequality in terms of the Kostlan distribution, obtaining that for all $s\geq 2DN$:
\beq\label{eq:pefe} \mathbb{P}\bigg\{ \|p\|_{\textrm{BW}}\geq s\cdot\mathrm{dist}_{\mathrm{BW}}(P, \Sigma_W)\bigg\}\leq C_3D N s^{-1}.\eeq
By Proposition \ref{lemma:discdegree}, we have $D\leq u d^u$. This implies that for some constants $a_1, a_2, a_3>0$ we have:
\beq 2DN\leq a_1d^{a_2}\quad \textrm{and}\quad C_3D N\leq a_3d^{a_2}.  \eeq
In particular \eqref{eq:pefe} finally implies that for all $s\geq a_1d^{a_2}$:
\beq \mathbb{P}\bigg\{\|p\|_{\mathrm{BW}}\leq s \cdot \mathrm{dist}_{\mathrm{BW}}(P, \Sigma_W)\bigg\}\geq 1-a_3d^{a_2}s^{-1}.\eeq
This finishes the proof.
\end{proof}

The following theorem estimates the probability that the stability event holds.\begin{theorem}[Probability estimation]\label{thm:approx}There exist constants $c_2, c_3, c_4>0$ (depending on $W$) such that for every $L\in\{0, \ldots , d\}$ with $d-L\in2\mathbb{N}$ and for every $\sigma>0$ we have
\begin{equation}\label{eq:lowerEL} \mathbb{P}(E_L)\geq 1-c_2\,d^{ c_3}L^{ c_4}e^{-\frac{L^2}{3d}}.\end{equation}
	\end{theorem}
	\begin{proof}
By Proposition \ref{prop1} there exists a constant $c=c(\ell,n,\nu)$ depending on $\ell,n,\nu$ such that
\begin{equation}\|p-p|_{L}\|_{C^{\ell+1}}^{\nu}\leq c d^{1/2} \|p-p|_{L}\|_{H^q}=(*)
\end{equation}
Moreover, observe that the proof of \cite[Proposition 2]{Antonio} works also for a polynomial list and gives the existence of a constant $c_2=c_2(n)>0$ (which depends on $n$) such that for all $t,q\geq 0$ and for every $L\in \{0, \ldots, d\}$ we have that
\begin{equation}\label{eq:pro2}(*)\leq cd^{1/2}t \|p\|_{\mathrm{BW}}\quad \text{ holds with probability } \quad\mathbb{P}_1\geq 1-c_2\frac{d^{-\frac{3n}{2}+1}L^{2q+n-2}e^{-\frac{L^2}{d}}}{t^2}.
\end{equation}
At the same time, by Proposition \ref{prop:discnhood}, for every $s\geq a_1d^{a_2}$:
\begin{equation}
\|p\|_{\textrm{BW}}\leq s \cdot \mathrm{dist}_{\mathrm{BW}}(P, \Sigma_W)\quad \text{ holds with probability } \quad
\mathbb{P}_2\geq 1-a_3\frac{d^{a_2}}{s}.
\end{equation}
Now, choosing $s, t$ as in the proof of \cite[Theorem 5]{Antonio}, there exist constants $c_1,c'_2, c'_3, c'_4>0$ such that for every $\sigma\geq 0$ we have
$\|p-p|_{L}\|_{C^{r+1}}<c_1\,\mathrm{dist}(P, \Sigma_W)$
with probability
\begin{equation}\label{eq:ff}\mathbb{P}(E_L)\geq 1-(1-\mathbb{P}_1)-(1-\mathbb{P}_2)\geq 1-\left(c'_2\,d^{c'_3}L^{c'_4}e^{-\frac{L^2}{d}}\sigma^2+\frac{1}{\sigma}\right).
\end{equation}
Let us denote now $\alpha(d, L)=c'_2\,d^{c'_3}L^{c'_4}e^{-\frac{L^2}{d}},$ so that we can rewrite \eqref{eq:ff} as:
\begin{equation} \mathbb{P}(E_L)\geq 1-\left(\alpha(d, L)\sigma^2+\frac{1}{\sigma}\right),\end{equation}
for all $\sigma>0$. In particular
\begin{equation} \mathbb{P}(E_L)\geq 1-\inf_{\sigma>0}\left(\alpha(d, L)\sigma^2+\frac{1}{\sigma}\right).\end{equation}
The above infimum is actually a minimum and it is reached at
$$ \sigma_0=3\left(\frac{\alpha(d, L)}{4}\right)^{1/3}.$$
Plugging the value of $\sigma_0$ into \eqref{eq:ff} gives the claim.
\end{proof}
The above estimate is very general, and one has to consider that we would like to have $L$ as small as possible and at the same time $\sigma$ as large as possible. We approach this issue as follows.

First, we choose $L$ to be of the order $O(\sqrt{d \log d})$ and we prove that with fast growing probability $p|_L$ stably approximates $p$; in particular, Theorem \ref{thm:quantstab} implies that for a given $W$, the $W$-type singular locus of a random Kostlan polynomial of degree $d$, with high probability as $d$ grows rapidly, is ambient isotopic to the $W$-type singular locus of a polynomial of degree $O(\sqrt{d\log d})$.
In the second step, we deal with exponential rarefactions of maximal configurations. More precisely, when choosing $L$ to be a root or a fraction of $d$, we can tune $\sigma$ so that the probability of stably approximate goes to 1 exponentially fast.

This distinction is the motivation for having three different regimes in Theorem \ref{thm:main}. Now we give a proof for this Theorem.

\begin{proof}[Proof of Theorem \ref{thm:main}]
Since $E_L\subseteq A_L$, and by Theorem \ref{thm:approx} we have
\beq \mathbb{P}(A_L)\geq\mathbb{P}(E_L)\geq 1- c_2d^{c_3}L^{c_4}e^{-\frac{L^2}{3d}}.\eeq
In the rest of the proof we plug in different values for $L$ and evaluate the right hand side of this inequality.
\emph{Proof of Theorem  \ref{thm:main}.1.} Let $L=b \sqrt{d\log d}$.  We see that with this choice
\begin{align}
\label{proof_thm:main_eq1}\mathbb{P}(A_L)&\geq 1- c_2d^{c_3}(b \sqrt{d\log d})^{c_4}e^{-\frac{b^2 \log d}{3}}\\
  & = 1- c_2b^{c_4}d^{c_3 + c_4/2 - b^2/3}(\log d)^{c_4/2}\\
  &\geq  1- c_2b^{c_4}d^{c_3 + c_4 - b^2/3}.
\end{align}
Thus, if $b>b_0:=\sqrt{3(c_3+c_4)}$, then we have
$$\mathbb{P}(A_L)\geq 1-a_1d^{-a_2},$$
where $a_1 = c_2b^{c_4}>0$ and $a_2 = -(c_3 + c_4 - b^2/3)>0$. This proves the first part of the theorem.

\emph{Proof of Theorem  \ref{thm:main}.2.} Let $L=d^b$ for some $b\in (1/2, 1).$ With this choice we have
\begin{align}
\mathbb{P}(A_L)&\geq 1- c_2d^{c_3 + bc_4}e^{-\frac{d^{2b-1}}{3}}  = 1- c_2 \frac{d^{c_3 + bc_4}}{e^{\frac{d^{2b-1}}{6}}} e^{-\frac{d^{2b-1}}{6}}.
\end{align}
Since $2b-1>0$, we have that $\frac{d^{c_3 + bc_4}}{e^{\frac{d^{2b-1}}{6}}}$ is bounded and so
$$\mathbb{P}(A_L)\geq 1-a_1e^{-d^{a_2}}$$
for some $a_1, a_2>0$ with $a_2<1$, because $2b-1<1$. This proves the second part of the theorem.

\emph{Proof of Theorem  \ref{thm:main}.3.} Let $L=bd$ with $b\in (0, 1)$. With this choice we have:
\begin{align}
\mathbb{P}(A_L)&\geq  1- c_2b^{c_4}d^{c_3+c_4}e^{-\frac{b^2d}{3}} = c_2b^{c_4}\frac{d^{c_3+c_4}}{e^{\frac{b^2d}{6}}}e^{-\frac{b^2d}{6}}.
\end{align}
We have that $\frac{d^{c_3+c_4}}{e^{\frac{b^2d}{6}}}$ is bounded and so
$$\mathbb{P}(A_L)\geq 1-a_1e^{-a_2d}$$
for some $a_1 , a_2>0$. This proves the third part of the theorem.

\emph{Proof of Theorem  \ref{thm:main}.4.}
Let $a>0$ and set $L=b\sqrt{d\log d}$. As in \eqref{proof_thm:main_eq1} we get
$$\mathbb{P}(A_L)\geq 1 - c_2b^{c_4}d^{c_3 + c_4 - b^2/3}.$$
If we pick $b$ and $d$ large enough, such that
$c_2b^{c_4}d^{c_3 + c_4 + a - b^2/3}\leq 1,$
then
$$\mathbb{P}(A_L)\geq 1 - d^{-a}.$$
This proves the fourth part of the theorem.
The proof of points (5) and (6) proceed along the same lines.
\end{proof}

%%%%%%%%%%%%%%%%%%%%%%%%%%%%%%%%%%%%%%%%%%%%%%%%%%%%%%%%%%%%
%%%%%%%%%%%%%%%%%%%%%%%%%%%%%%%%%%%%%%%%%%%%%%%%%%%%%%%%%%%%
%%%%%%%%%%%%%%%%%%%%%%%%%%%%%%%%%%%%%%%%%%%%%%%%%%%%%%%%%%%%
%%%%%%%%%%%%%%%%%%%%%%%%%%%%%%%%%%%%%%%%%%%%%%%%%%%%%%%%%%%%
%%%%%%%%%%%%%%%%%%%%%%%%%%%%%%%%%%%%%%%%%%%%%%%%%%%%%%%%%%%%
%%%%%%%%%%%%%%%%%%%%%%%%%%%%%%%%%%%%%%%%%%%%%%%%%%%%%%%%%%%%
%%%%%%%%%%%%%%%%%%%%%%%%%%%%%%%%%%%%%%%%%%%%%%%%%%%%%%%%%%%%
%%%%%%%%%%%%%%%%%%%%%%%%%%%%%%%%%%%%%%%%%%%%%%%%%%%%%%%%%%%%

\section{Maximal Configurations are rare events}\label{sec:maximal_is_rare}
In this section we provide a couple of examples of tail probabilities, proving  exponential rarefaction of maximal configurations. For a semialgebraic set $Y$, we denote by $b(Y)$ the sum of its Betti numbers (this sum is finite by semialgebraicity).

\begin{proposition}[Exponential rarefaction of maximal configurations]\label{thm:rare}Let $W\subseteq J^r(S^n, \R^m)$ be a singularity type. For every $C>0$ there exist $a_1, a_2>0$ such that
\beq \mathbb{P}\left(\left\{ p\in \mathcal{P}_{n, \underline{d}}\,|\,b(j^rp^{-1}(W))\geq Cd^n \right\}\right)\leq a_1e^{-a_2d}.\eeq
\end{proposition}
\begin{proof}Let $W\subseteq J^{r}(S^n, \R^m)$. We recall that if $q:S^n\to \R^m$ is a polynomial map with degree bounded by $L$ such that $j^rq\pitchfork W$, the following estimate is proved in \cite{LerarioStecconi} for the sum of the Betti numbers of $j^rq^{-1}(W)$:
\beq \label{eq:estimated}b(j^rq^{-1}(W))\leq C_1 L^n.\eeq
In particular we see that if for $p\in \mathcal{P}_{n, \underline{d}}$ we have $j^rp\pitchfork W$ (which happens with probability 1 by \cite[Theorem 1, point (4)]{LerarioStecconi}) and $b(j^{r}p^{-1}(W))\geq C d^n$, then if moreover $p\in A_L$ (the stability event from Definition \ref{def:stabilityevent}) we must have $L\geq \left(\frac{C}{C_1}\right)^{1/n}d$. Set now:
\beq L=\min\left\{d, \left\lfloor \left(\frac{C}{C_1}\right)^{1/n}d\right\rfloor+1\right\}.\eeq
Then for every $p\in A_L$, we have
\beq b(j^rp^{-1}(W))=b(j^rp|_L^{-1}(W))< Cd^n.\eeq
In particular, we must have
\beq\label{eq:cont} \left\{b(j^rp^{-1}(W)\geq Cd^n\right\} \subseteq (A_L)^c\eeq
(the complement of the stability event $A_L$).
Point (3) of Theorem \ref{thm:main} implies that there exists constants $a_1, a_2>0$ such that $\mathbb{P}(A_L)\geq 1-a_1e^{-a_2d}$, which combined  with \eqref{eq:cont} gives:
\beq \mathbb{P}\left(\left\{b(j^rp^{-1}(W)\geq Cd^n\right\}\right)\leq a_1e^{-a_2d}.
\eeq
This proves the proposition.
\end{proof}

We can apply the previous result to the examples discussed in the Introduction.
\begin{enumerate}
\item In the case $W=S^n\times \{0\}\subset J^{0}(S^n, \R^m),$  Harnack's bound implies that $$b(j^0p^{-1}(W))=b(\{p=0\})\leq O(d^n)$$ and the probability that $b(\{p=0\})\geq C_3d^n$ is exponentially small. This example is extensively discussed in \cite{Antonio}.
\item In the case $W=\{j^1f=0\}\subset J^1(S^n, \R)$ then $j^1f^{-1}(W)$ is the set of critical points of $p:S^n\to \R$;  By \cite{CS} for a polynomial $P$ of degree $d$
\begin{equation}
\#j^1p^{-1}(W)\leq 2(d-1)^n+\dots+(d-1)+1
\end{equation}
(this bounds follows from complex algebrais geometry), and this estimate was recently proved to be sharp by Kozhasov \cite{Kozhasov}.
In this context \cite[Theorem 14]{LerarioStecconi} tells that the expected number of critical points is of the order $\Theta(d^{n/2})$ and Theorem \ref{thm:rare} implies that for every $C>0$ there exists $a_1, a_2>0$ such that the probability that a random polynomial of degree $d$ has at least $C d^{n}$ many critical points is smaller than $a_1 e^{-a_2 d}$.
\item $W=\{\mathrm{d}f=0, \mathrm{d}^2f>0\}\subset J^2(S^n, \R),$ and $j^2f^{-1}(W)$ is the set of non-degenerate \emph{minima} of $f:S^n\to \R$; the same argument can be applied here: the expectation of the number of minima is of the order $\Theta(d^{n/2})$ (\cite[Theorem 14]{LerarioStecconi}). There are polynomials of degree $d$ with $\Theta(d^{n})$ many non-degenerate minima, but the measure of the sets of such polynomials becomes exponentially small ad $d$ grows.
\item $W=\{\textrm{Whitney cusps}\}\subset J^3(S^2, \R^2),$ and $j^{3}f^{-1}(W)$ is the set of points where $f:S^2\to \R^2$ has critical points of type Whitney cusp, i.e. in some coordinates near this point the map has the form $(x_1, x_2)\mapsto(x_1, x_2^3-x_1x_2)$, see \cite{Callahan} (we need a condition on the third jet in order to make sure that this local form exists). The number of such cusps, for a polynomial map of degree $d$ is $O(d^2)$. On average there are $\Theta(d)$ many of them and the probability of having $\Theta(d^2)$ Whitney cusps becomes exponentially small as $d$ grows.
\end{enumerate}
\begin{remark}By applying points (1) and (2) from Theorem \ref{thm:main} one can obtain other similar tail probabilities. For example, choosing $L=b\sqrt{d\log d}$ as in Theorem \ref{thm:main} we get that
\beq b(Z(p))\leq O((d\log d)^{n/2})\quad \textrm{with $\mathbb{P}\to 1$ as $d\to \infty$}.\eeq
This result follows using the same ideas as before, combined with the deterministic estimate $b(Z(Q))\leq O(L^n)$ for the sum of the Betti numbers of the zero set of a polynomial $Q$ of degree $L$ in projective space (see \cite{Milnor}).
\end{remark}

\bibliographystyle{alpha}
\bibliography{stabilityrandompoly}

\end{document}